\documentclass[a4paper, 12 pt, oneside, reqno]{amsart} 
\usepackage{amsfonts, amssymb, amsmath, eucal, amsthm}
\usepackage[all]{xy}
\usepackage{tikz}
\usetikzlibrary{matrix}

\newtheorem{lemma}{Lemma}
\numberwithin{lemma}{section}
\newtheorem{theorem}[lemma]{Theorem}
\newtheorem*{theorem*}{Theorem}
\newtheorem*{corollary*}{Corollary}
\newtheorem{proposition}[lemma]{Proposition}
\newtheorem{corollary}[lemma]{Corollary}

\newtheorem*{conjecture}{Conjecture}
\newtheorem{Th}{Theorem}

\theoremstyle{definition}
\newtheorem{definition}[lemma]{Definition}

\newtheorem{example}[lemma]{Example}
\newtheorem{remark}[lemma]{Remark}

\newtheorem*{remark*}{Remark}

\newcommand{\A}{\mathcal{A}}
\newcommand{\F}{\mathbb{F}}
\newcommand{\B}{\mathcal{B}}
\newcommand{\calF}{\mathcal{F}}

\renewcommand{\S}{\mathcal{S}}
\newcommand{\Z}{\mathbb{Z}}

\newcommand{\SC}{\mathit{SC}}
\newcommand{\SP}{\mathit{SP}}
\newcommand{\GL}{\mathit{GL}}
\newcommand{\SL}{\mathit{SL}}
\newcommand{\Aut}{\mathrm{Aut}}
\DeclareMathOperator{\colim}{colim}
\DeclareMathOperator{\rank}{rank}
\DeclareMathOperator{\tor}{Tor}
\renewcommand{\diamond}{\oplus}

\usepackage{geometry}

\title{On the edge of the stable range}
\author{Richard Hepworth}
\address{Institute of Mathematics\\
University of Aberdeen\\
Aberdeen AB24 3UE\\
United Kingdom}
\email{r.hepworth@abdn.ac.uk}


\begin{document}

\begin{abstract}
We prove a general homological stability theorem
for certain families of groups equipped with product maps,
followed by two theorems of a new kind that give
information about the last two homology groups outside the stable range.
(These last two unstable groups are the `edge' in our title.)
Applying our results to automorphism groups of free groups yields 
a new proof of homological stability with an improved stable range, 
a description of the last unstable group up to a single ambiguity,
and a lower bound on the rank of the penultimate unstable group.
We give similar applications to the general linear groups of the integers
and of the field of order 2, this time recovering the known stablility range.
The results can also be applied to general linear groups of
arbitrary principal ideal domains, symmetric groups, and braid groups.
Our methods require us to use field coefficients throughout.
\end{abstract}

\maketitle

\section{Introduction}

A sequence of groups and inclusions
$
	G_1\hookrightarrow G_1\hookrightarrow G_3\hookrightarrow\cdots
$
is said to satisfy \emph{homological stability}
if in each degree $d$ there is an integer $n_d$ such
that the induced map $H_d(G_{n-1})\to H_d(G_n)$
is an isomorphism for $n> n_d$.
Homological stability is known to hold for many families of groups,
including symmetric groups~\cite{Nakaoka},
general linear groups~\cite{Quillen, Charney, vanderKallen},
mapping class groups of surfaces and 
3-manifolds~\cite{Harer, RandalWilliamsMCG, WahlMCG, HatcherWahl},
diffeomorphism groups of highly connected 
manifolds~\cite{GalatiusRandalWilliams}, and automorphism groups of 
free groups~\cite{HatcherVogtmannStability,HatcherVogtmannRational}.
Homological stability statements often also specify
that the last map outside the range $n> n_d$ is a surjection,
so that the situation can be pictured as follows.
\[
	\cdots
	\to
	H_d(G_{n_d-3})
	\to
	\underbrace{
	H_d(G_{n_d-2})
	\to
	H_d(G_{n_d-1})
	}_{\text{edge of the stable range}}
	\twoheadrightarrow
	\underbrace{
	{H_d(G_{n_d})}
	\xrightarrow{\cong}
	{H_d(G_{n_d+1})}
	\xrightarrow{\cong}
	\cdots}_{\text{stable range}}
\]

The groups $H_d(G_{n_d}), H_d(G_{n_d+1}),\ldots$, which are all isomorphic,
are said to form the \emph{stable range}.
This paper studies what happens at
\emph{the edge of the stable range}, by which we mean the
last two unstable groups $H_d(G_{n_d-2})$ and $H_d(G_{n_d-1})$.
We prove a new and rather general homological stability result
that gives exactly the picture above with $n_d=2d+1$.
Then we prove two theorems of an entirely new kind.
The first describes the {kernel} of the surjection 
$H_d(G_{n_d-1})\twoheadrightarrow H_d(G_{n_d})$,
and the second explains how to make the map
$H_d(G_{n_d-2})\to H_d(G_{n_d-1})$ into a {surjection}
by adding a new summand to its domain.
These general results hold for homology with coefficients in
an arbitrary field.

We apply our general results to
general linear groups of principal ideal domains (PIDs)
and automorphism groups of free groups.
In both cases 
we obtain new proofs of homological stability,
recovering the known stable range for the general linear groups,
and improving upon the known stable range for $\Aut(F_n)$.
We also obtain new information on the last two unstable homology
groups for $\Aut(F_n)$, $\GL_n(\Z)$ and $\GL_n(\F_2)$,
in each case identifying the last unstable group up to a single
ambiguity.

Our proofs follow an overall pattern that is familiar in homological stability.
We define a sequence of complexes acted on by
the groups in our family, and we assume that they satisfy
a connectivity condition.
Then we use an algebraic argument, based on spectral sequences
obtained from the actions on the complexes, to deduce the result.
The connectivity condition has to be verified separately for
each example, but it turns out that in our examples 
the proof is already in the literature, or can be deduced from it.
The real novelty in our paper is the algebraic argument.
To the best of our knowledge it has not been used before,
either in the present generality or in any specific instances.
Even in the case of general linear groups of PIDs,
where our complexes are exactly the ones used by Charney in the original
proof of homological stability~\cite{Charney} for Dedekind domains,
we are able to improve the stable range obtained, matching the best known.

\subsection{General results}
Let us state our main results, after first establishing
some necessary terminology.
From this point onwards homology is to be taken with coefficients
in an arbitrary field $\F$, unless stated otherwise.

A \emph{family of groups with multiplication}
$(G_p)_{p\geqslant 0}$ consists of a sequence of groups
$G_0,G_1,G_2,\ldots$  equipped with product maps $G_p\times G_q\to G_{p+q}$
for $p,q\geqslant 0$, subject to some simple axioms.
See section~\ref{families-section} for the precise definition.
The axioms imply in particular that 
$\bigoplus_{p\geqslant 0}H_\ast(G_p)$ is a graded commutative ring.
Examples include the symmetric groups, braid groups,
the general linear groups of a PID,
and automorphism groups of free groups.

To each family of groups with multiplication $(G_p)_{p\geqslant 0}$
we associate the \emph{splitting posets} $\SP_n$ for $n\geqslant 2$.  
If we think of $G_n$ as the group of symmetries of an `object of size $n$',
then an element of $\SP_n$ is a splitting of that object
into two ordered nontrivial pieces.
See section~\ref{section-spn} for the precise definition.
The \emph{stabilisation map}
$
	s_\ast\colon H_\ast(G_{n-1})\to H_\ast(G_n)
$
is the map induced by the homomorphism
$G_{n-1}\to G_n$ that takes the product on the left with the neutral element
of $G_1$.  Our first main result is the following 
homological stability theorem.

\begin{Th}\label{theorem-stability}
Let $(G_p)_{p\geqslant 0}$ be a family of groups with 
multiplication, and assume that $|\SP_n|$ is $(n-3)$-connected
for all $n\geqslant 2$.  Then the stabilisation map
\[
	s_\ast\colon H_\ast(G_{n-1})\longrightarrow H_\ast(G_n)
\]
is an isomorphism for $\ast\leqslant\frac{n-2}{2}$ and a
surjection for $\ast\leqslant\frac{n-1}{2}$.
Here homology is taken with coefficients in an
arbitrary field.
\end{Th}

Theorem~\ref{theorem-stability} overlaps with work in progress of
S\o ren Galatius, Alexander Kupers and Oscar Randal-Williams.
Indeed, if we were to add the assumption that $\bigsqcup_{p\geqslant 0} G_p$
is a braided monoidal groupoid, then it
would follow from the work of Galatius, Kupers and Randal-Williams.
(The definition of family of groups with multiplication ensures that
$\bigsqcup_{p\geqslant 0}G_p$ is a monoidal groupoid;
the braiding assumption holds in all of our examples.)
We will mention other points of overlap as they occur.

In a given degree $m$, Theorem~\ref{theorem-stability}
gives us the surjection and isomorphisms
in the following sequence.
\[
	\cdots
	\to
	H_m(G_{2m-2})
	\to
	\underbrace{
	H_m(G_{2m-1})
	\to
	H_m(G_{2m})
	}_{\text{edge of the stable range}}
	\twoheadrightarrow
	\underbrace{
	{H_m(G_{2m+1})}
	\xrightarrow{\cong}
	{H_m(G_{2m+2})}
	\xrightarrow{\cong}
	\cdots}_{\text{stable range}}
\]
Our next two theorems extend
into the {edge of the stable range}.

\begin{Th}\label{theorem-kernel}
Let $(G_p)_{p\geqslant 0}$ be a family of groups with 
multiplication, and assume that $|\SP_n|$ is $(n-3)$-connected
for all $n\geqslant 2$.  Then the kernel of the map
\[
	s_\ast\colon H_m(G_{2m})\twoheadrightarrow H_m(G_{2m+1})
\]
is the image of the product map
\[
	H_{1}(G_{2})^{\otimes {m-1}}
	\otimes\ker[H_1(G_2)\xrightarrow{s_\ast} H_1(G_3)]
	\longrightarrow
	H_{m}(G_{2m}).
\]
Here homology is taken with coefficients in an
arbitrary field.
\end{Th}

\begin{Th}\label{theorem-realisation}
Let $(G_p)_{p\geqslant 0}$ be a family of groups with 
multiplication, and assume that $|\SP_n|$ is $(n-3)$-connected
for all $n\geqslant 2$.  Then the map 
\[
	H_m(G_{2m-1})\oplus H_1(G_2)^{\otimes m}
	\twoheadrightarrow H_m(G_{2m})
\]
is surjective.
Here homology is taken with coefficients in an
arbitrary field.
\end{Th}

Homological stability results like Theorem~\ref{theorem-stability}
are often combined with theorems computing
the stable homology $\lim_{n\to\infty}H_\ast(G_n)$
to deduce the value of $H_\ast(G_n)$ in the stable range.
In a similar vein, Theorems~\ref{theorem-kernel} and~\ref{theorem-realisation}
allow us to bound the last two unstable groups $H_m(G_{2m})$ and $H_m(G_{2m-1})$
in terms of $\lim_{n\to\infty}H_\ast(G_n)$.
In the following subsections we will see how this works for 
automorphism groups of free groups
and general linear groups of PIDs.
Note that our results do not rule out the possibility of a larger stable
range than the one provided by Theorem~\ref{theorem-stability}.
Nevertheless, in what follows we will refer to $H_m(G_{2m})$
and $H_m(G_{2m-1})$ as the `last two unstable groups'.

\subsection{Applications to automorphism groups of free groups}
The automorphism groups of free groups form a family of groups
with multiplication $(\Aut(F_n))_{n\geqslant 0}$.
In this case the splitting poset $\SP_n$ consists of pairs 
$(A,B)$ of proper subgroups of $F_n$
satisfying $A\ast B = F_n$.  
By relating the splitting poset to the poset of free factorisations
studied by Hatcher and Vogtmann in~\cite{HatcherVogtmannCerf},
we are able to show that $|\SP_n|$ is $(n-3)$-connected,
so that Theorems~\ref{theorem-stability},
\ref{theorem-kernel} and~\ref{theorem-realisation}
can be applied.
Our first new result is obtained using Theorem~\ref{theorem-stability}
in arbitrary characteristic, and Theorems~\ref{theorem-stability},
\ref{theorem-kernel} and~\ref{theorem-realisation}
in characteristic other than $2$.

\begin{Th}\label{theorem-autfn-stability}
Let $\F$ be a field.  Then the stabilisation map
\[
	s_\ast\colon H_\ast(\Aut(F_{n-1});\F)
	\longrightarrow
	H_\ast(\Aut(F_n);\F)
\]
is an isomorphism for $\ast\leqslant\frac{n-2}{2}$ and 
a surjection for $\ast\leqslant\frac{n-1}{2}$.
Moreover, if $\mathrm{char}(\F)\neq 2$, then
$s_\ast$ is an isomorphism for $\ast\leqslant\frac{n-1}{2}$
and a surjection for $\ast\leqslant\frac{n}{2}$.
\end{Th}

Hatcher and Vogtmann showed in~\cite{HatcherVogtmannStability}
that $s_\ast\colon H_\ast(\Aut(F_{n-1}))\to H_\ast(\Aut(F_n))$
is an isomorphism for $\ast\leqslant \frac{n-3}{2}$ and a surjection
for $\ast\leqslant\frac{n-2}{2}$, where homology is taken with arbitrary
coefficients.
Theorem~\ref{theorem-autfn-stability} increases this
stable range one step to the left in each degree
when coefficients are taken in a field,
and two steps to the left in each degree
when coefficients are taken in a field
of characteristic other than $2$.
(In characteristic $0$ this falls far short of the 
best known result~\cite{HatcherVogtmannRational}.)
In particular we learn for the first time that 
the groups $H_m(\Aut(F_{2m+1});\F)$ are stable.

By applying
Theorems~\ref{theorem-kernel} and~\ref{theorem-realisation}
when $\F=\F_2$, we are able to learn the following about 
the last two unstable groups
$H_m(\Aut(F_{2m});\F_2)$ and $H_m(\Aut(F_{2m-1});\F_2)$.

\begin{Th}\label{theorem-autfn-outside}
Let $t\in H_1(\Aut(F_2);\F_2)$ denote the element
determined by the transformation
$x_1\mapsto x_1$, $x_2\mapsto x_1x_2$,
and let $m\geqslant 1$.
Then the kernel of the stabilisation map
\[
	s_\ast\colon H_m(\Aut(F_{2m});\F_2)\twoheadrightarrow 
	H_m(\Aut(F_{2m+1});\F_2)
\]
is the span of $t^m$, and the map
\[
	H_m(\Aut(F_{2m-1});\F_2)\oplus\F_2\to H_m(\Aut(F_{2m});\F_2),
	\qquad
	(x,y)\mapsto s_\ast(x)+y\cdot t^m
\]
is surjective.
\end{Th}

This theorem shows that the last unstable group
$H_m(\Aut(F_{2m});\F_2)$ is either isomorphic to the stable homology
$\lim_{n\to\infty} H_m(\Aut(F_n);\F_2)$,
or is an extension of it by a copy of $\F_2$ generated by $t^m$.
It does not state which possibility holds.
Galatius~\cite{Galatius} identified the stable homology 
$\lim_{n\to\infty}H_\ast (\Aut(F_n))$ 
with $H_\ast(\Omega^\infty_0 S^\infty)$,
where $\Omega_0^\infty S^\infty$ denotes a path-component of 
$\Omega^\infty S^\infty = \colim_{n\to\infty} \Omega^n S^n$.
Thus we are able to place the following bounds on the dimensions
of the last two unstable groups for $m\geqslant 1$,
where $\epsilon$ is either $0$ or $1$.
\begin{gather*}
	\dim(H_m(\Aut(F_{2m});\F_2))
	=
	\dim(H_m(\Omega_0^\infty S^\infty;\F_2))+\epsilon
	\\
	\dim(H_m(\Aut(F_{2m-1});\F_2))\geqslant
	\dim(H_m(\Omega_0^\infty S^\infty;\F_2))
\end{gather*}

\subsection{Applications to general linear groups of PIDs}

The general linear groups of a commutative ring $R$ form a family
of groups with multiplication $(\GL_n(R))_{n\geqslant 0}$.
When $R$ is a PID,  
the realisation $|\SP_n|$ of the splitting
poset is precisely the \emph{split building} $[R^n]$ 
studied by Charney, who showed that it is $(n-3)$-connected~\cite{Charney}.
Theorems~\ref{theorem-stability},
\ref{theorem-kernel} and~\ref{theorem-realisation}
can therefore be applied in this setting.

Theorem~\ref{theorem-stability} shows that 
$H_\ast(\GL_{n-1}(R))\to H_\ast(\GL_n(R))$
is onto for $\ast\leqslant\frac{n-1}{2}$ and an isomorphism
for $\ast\leqslant\frac{n-2}{2}$, where homology is taken with
field coefficients.  This exactly recovers homological stability 
with the range due to van der Kallen~\cite{vanderKallen},
but only with field coefficients.
Theorems~\ref{theorem-kernel} and~\ref{theorem-realisation}
then allow us to learn about the last two unstable groups
$H_{m}(\GL_{2m-1}(R))$ and $H_m(\GL_{2m}(R))$, where  
little seems to be known in general.
In order to illustrate this we specialise to the cases $R=\Z$ and
$R=\F_2$ and take coefficients in $\F_2$; this is the content
of our next two subsections.

\subsection{Applications to the general linear groups of $\Z$}

We now specialise to the groups $\GL_n(\Z)$ and take coefficients in $\F_2$.
Theorems~\ref{theorem-kernel} and~\ref{theorem-realisation}
give us the following information about the final two unstable groups
$H_m(\GL_{2m}(\Z);\F_2)$ and $H_m(\GL_{2m-1}(\Z);\F_2)$.

\begin{Th}
\label{glnz-outside}
Let $t$ denote the element of $H_1(\GL_2(\Z);\F_2)$ determined
by the matrix
$\left(\begin{smallmatrix} 1 & 1 \\ 0 & 1\end{smallmatrix}\right)$
and let $m\geqslant 1$.
Then the kernel of the stabilisation map
\[
	s_\ast\colon H_m(\GL_{2m}(\Z);\F_2)\twoheadrightarrow 
	H_m(\GL_{2m+1}(\Z);\F_2)
\]
is the span of $t^m$, and the map
\[
	H_m(\GL_{2m-1}(\Z);\F_2)\oplus\F_2\to H_m(\GL_{2m}(\Z);\F_2),
	\qquad
	(x,y)\mapsto s_\ast(x)+y\cdot t^m
\]
is surjective.
\end{Th}

This theorem shows that the last unstable group
$H_m(\GL_{2m}(\Z);\F_2)$ is either isomorphic to 
the stable homology $\lim_{n\to\infty} H_m(\GL_m(\Z);\F_2)$,
or is an extension of it by a copy of $\F_2$
generated by $t^m$.  It does not guarantee that $t^m\neq 0$,
and so does not specify which possibility occurs.
The theorem also gives us the following lower bounds 
on the dimensions of the last two
unstable groups in terms of $\dim(\lim_{n\to\infty} H_m(\GL_n(\Z);\F_2))$,
and in particular shows that they are highly nontrivial.
\begin{gather*}
	\dim(H_m(\GL_{2m}(\Z);\F_2))
	=
	\dim\left(\lim_{n\to\infty}H_m(\GL_n(\Z);\F_2)\right)+\epsilon
	\\
	\dim(H_m(\GL_{2m-1}(\Z);\F_2)) 
	\geqslant
	\dim\left(\lim_{n\to\infty} H_m(\GL_n(\Z);\F_2)\right)
\end{gather*}
Here $\epsilon$ is either $0$ or $1$.

\subsection{Applications to the general linear groups of $\F_2$}

Now let us specialise to the groups $\GL_n(\F_2)$.
Quillen showed that in this case the stable homology 
$\lim_{n\to\infty}H_\ast(\GL_n(\F_2);\F_2)$ vanishes~\cite[Section~11]{Quillen}.
Combining this with Maazen's stability result shows that 
$H_m(\GL_n(\F_2);\F_2)=0$ for $n\geqslant 2m+1$.
It is natural to ask for a description of the
final unstable homology groups
$H_m(\GL_{2m}(\F_2);\F_2)$.  These are known to be 
nontrivial for $m=1$ and $m=2$, the latter case
being due to Milgram and Priddy (Example~2.6 and Theorem~6.5
of~\cite{MilgramPriddy}), but to the best of our knowledge nothing
further is known.  By applying Theorem~\ref{theorem-kernel}
we obtain the following result, which determines each of the groups
$H_m(\GL_{2m}(\F_2);\F_2)$ up to a single ambiguity.

\begin{Th}\label{theorem-glnftwo}
Let $t$ denote the element of $H_1(\GL_2(\F_2);\F_2)$ determined
by the matrix 
$\left(\begin{smallmatrix} 1 & 1 \\ 0 & 1\end{smallmatrix}\right)$.
Then $H_m(\GL_{2m}(\F_2);\F_2)$ is either trivial, or is a copy
of $\F_2$ generated by the class $t^m$.
\end{Th}

We hope that by extending the techniques of the present paper
we will be able in future to prove the following conjecture.
We anticipate that the known non-vanishing of $t$ and $t^2$ will be
an essential ingredient in its proof.

\begin{conjecture}
For every $m\geqslant 1$ the group 
$H_m(\GL_{2m}(\F_2);\F_2)$ is a single copy of $\F_2$
generated by the class $t^m$, where $t\in H_1(\GL_2(F_2);\F_2)$
is the element determined by the matrix 
$\left(\begin{smallmatrix} 1 & 1 \\ 0 & 1\end{smallmatrix}\right)$.
\end{conjecture}

A proof of this conjecture would, \emph{via} the homomorphisms
$\Aut(F_n)\to\GL_n(\Z)\to\GL_n(\F_2)$, also resolve the ambiguities
in Theorems~\ref{theorem-autfn-outside} and~\ref{glnz-outside},
showing that the final unstable homology groups
$H_m(\Aut(F_{2m}),\F_2)$ and $H_m(\GL_{2m}(\Z);\F_2)$ are extensions by $\F_2$
of $\lim_{n\to\infty} H_m(\Aut(F_{n});\F_2)$ 
and $\lim_{n\to\infty}H_m(\GL_{n}(\Z);\F_2)$
respectively.  In particular this would confirm
that the known homological stability ranges are sharp.

Theorem~\ref{theorem-glnftwo} is relevant to questions 
about the groups $H^m(\GL_{2m}(\F_2);\F_2)$ raised by Milgram
and Priddy in~\cite[p.301]{MilgramPriddy}, and posed explicitly
by Priddy in~\cite[section 5]{PriddyProblem}.  Let $M_{mm}$ denote the subgroup
of $\GL_{2m}(\F_2)$ consisting of matrices of the form
\[
	\begin{pmatrix}
		I_m & \ast
		\\
		0 & I_m
	\end{pmatrix}.
\]
Milgram and Priddy describe an element $\det_m\in H^m(M_{mm};\F_2)$
that is invariant under the action of $N_{\GL_{2m}(\F_2)}(M_{mm})/M_{mm}
=\GL_m(\F_2)\times\GL_m(\F_2)$,
and so potentially lifts to an element of $H^m(\GL_{2m}(\F_2);\F_2)$.
Priddy asks whether $\det_m$ lifts to $H^m(\GL_{2m}(\F_2);\F_2)$, 
and if so, whether it spans $H^m(\GL_{2m}(\F_2);\F_2)$.
As explained to us by David Sprehn, $t^m$ is the image of a class
in $H_m(M_{mm};\F_2)$, 
and $\det_m$ spans the invariants 
$H^m(M_{mm};\F_2)^{\GL_m(\F_2)\times\GL_m(\F_2)}$.
Theorem~\ref{theorem-glnftwo} therefore shows that
$H^m(\GL_{2m}(\F_2);\F_2)$ is either trivial, or is a single copy
of $\F_2$ generated by a lift of $\det_m$.

\subsection{Decomposability beyond the stable range}

Let $(G_p)_{p\geqslant 0}$ be a family of groups with multiplication,
and consider the bigraded commutative ring 
$A=\bigoplus_{p\geqslant 0}H_\ast(G_p)$.
Homological stability tells us that any element of $H_\ast(G_p)$
that lies in the stable range decomposes as a product of elements in the
augmentation ideal of $A$.  (In fact it tells us that such an element
decomposes as a product with the generator of $H_0(G_1)$.)
We believe that connectivity bounds on the splitting complex
can yield decomposability results far beyond the stable range.
The following conjecture was formulated after studying explicit
computations for symmetric groups and braid groups~\cite{CLM},
in which cases it holds.

\begin{conjecture}
\label{conjecture}
Let $(G_p)_{p\geqslant 0}$ be a family of groups with multiplication.
Suppose that $|\SP_n|$ is $(n-3)$-connected for all $n\geqslant 2$.
Then the map
\[
	\mu\colon
	\bigoplus_{\substack{p+q=n\\p,q\geqslant 1}}
	H_\ast(G_p)\otimes H_\ast(G_{q})
	\longrightarrow
	H_\ast(G_n)
\]
is surjective in degrees $\ast\leqslant (n-2)$,
and its kernel is the image of 
\[
	\alpha\colon
	\bigoplus_{\substack{p+q+r=n\\p,q,r\geqslant 1}}
	H_\ast(G_p)\otimes H_\ast(G_q)\otimes H_\ast(G_r)
	\longrightarrow
	\bigoplus_{\substack{p+q=n\\p,q\geqslant 1}}
	H_\ast(G_p)\otimes H_\ast(G_{q})
\]
in degrees $\ast\leqslant (n-3)$.
Here $\mu$ and $\alpha$ are defined by $\mu(x\otimes y) = x\cdot y$
and $\alpha(x\otimes y\otimes z) = (x\cdot y)\otimes z - x\otimes(y\cdot z)$.
\end{conjecture}

We are able to prove the surjectivity statement in degrees 
$\ast\leqslant \frac{n}{2}$ and the injectivity statement
in degrees $\ast\leqslant \frac{n-1}{2}$, both of which
are half a degree better than the stable range
(Lemmas~\ref{twomplusoneacyclic-lemma} and~\ref{twomacyclic-lemma}),
and Theorems~\ref{theorem-kernel} and~\ref{theorem-realisation}
are the `practical' versions of these facts.
We hope that in future work we will be able to 
obtain information further beyond the stable range.

\subsection{Organisation of the paper}

In the first half of the paper we introduce the concepts required
to understand the statements of 
Theorems~\ref{theorem-stability}, \ref{theorem-kernel}
and~\ref{theorem-realisation} and then, assuming these theorems
for the time being, we give the proofs of the applications
stated earlier in this introduction.
Section~\ref{families-section} introduces families of groups with
multiplication, and introduces four main examples: the symmetric groups,
general linear groups of PIDs, automorphism groups of free groups,
and braid groups.
Section~\ref{section-spn} introduces the splitting posets
$\SP_n$
associated to a family of groups with multiplication,
and identifies them in the four examples.
In section~\ref{section-connectivity} we show that
for these four examples, the realisation $|\SP_n|$
of the splitting poset is $(n-3)$-connected.
Finally, in section~\ref{section-applications}
we give the proofs of Theorems~\ref{glnz-outside},
\ref{theorem-glnftwo}, \ref{theorem-autfn-stability}
and \ref{theorem-autfn-outside}.

In the second half of the paper we give the proofs of 
our three general results,
Theorems~\ref{theorem-stability}, \ref{theorem-kernel}
and~\ref{theorem-realisation}.  Section~\ref{section-scn}
introduces the splitting complex, an alternative to the splitting
poset that features in the rest of the argument.
Section~\ref{section-bar} introduces a graded chain complex
$\B_n$ obtained from a family of groups with multiplication.
In section~\ref{section-spectral} we show that, under the hypotheses of
Theorems~\ref{theorem-stability}, \ref{theorem-kernel}
and~\ref{theorem-realisation} there is a spectral sequence with $E^1$-term
$\B_n$ and converging to $0$ in total degrees $\leqslant(n-2)$.
Section~\ref{section-filtration} introduces
and studies a filtration on $\B_n$.  The filtration allows us
to understand the homology of $\B_n$ inductively within a range
of degrees.  Then sections~\ref{section-stability}, \ref{section-realisation}
and~\ref{section-kernel} give the proofs of the three theorems.

\subsection{Acknowledgements}
\label{acknowledgements}

My thanks to Rachael Boyd, Anssi Lahtinen, Martin
Palmer, Oscar Randal-Williams, David Sprehn 
and Nathalie Wahl for useful discussions.

\section{Families of groups with multiplication}
\label{families-section}

In this section we define the families of groups with multiplication
to which our methods will apply, and we provide a series of examples.

\begin{definition}\label{definition-families}
A \emph{family of groups with multiplication} $(G_p)_{p\geqslant 0}$
is a sequence of discrete groups $G_0,G_1,G_2,\ldots$
equipped with a \emph{multiplication map}
\[
	G_p\times G_q\longrightarrow G_{p+q},
	\qquad
	(g,h)\longmapsto g\diamond h
\]
for each $p,q\geqslant 0$. 
We assume that the following axioms hold:
\begin{enumerate}
	\item \emph{Unit:}
	The group $G_0$ is the trivial group, and its unique
	element $e_0$ acts as a unit for left and right multiplication.
	In other words
	$e_0\diamond g = g = g\diamond e_0$
	for all $p\geqslant 0$ and all $g\in G_p$.

	\item \emph{Associativity:}
	The associative law
	\[
		(g\diamond h)\diamond k = g\diamond(h\diamond k).
	\]
	holds for all $p,q,r\geqslant 0$ and
	all $g\in G_p$, $h\in G_q$ and $k\in G_r$.
	Consequently, for any sequence $p_1,\ldots,p_r\geqslant 0$
	there is a well-defined \emph{iterated multiplication map}
	\[
		G_{p_1}\times \cdots \times G_{p_r}\longrightarrow
		G_{p_1+\cdots+p_r}.
	\]	

	\item \emph{Commutativity:}
	The product maps are commutative up to conjugation,
	in the sense that there exists an
	element $\tau_{pq}\in G_{p+q}$ such that the squares 
	\[\xymatrix{
		G_p\times G_q
		\ar[r]\ar[d]_\cong
		&
		G_{p+q}
		\ar[d]^{c_{\tau_{pq}}}
		\\
		G_q\times G_p
		\ar[r]
		&
		G_{p+q}
	}\]
	commute, where $c_{\tau_{pq}}$ denotes conjugation by $\tau_{pq}$.
	(We do not impose any further conditions
	upon the $\tau_{pq}$.)

	\item \emph{Injectivity:}
	The multiplication maps are all injective.
	It follows that the iterated multiplication maps are also injective.
	Using this, we henceforth regard 
	$G_{p_1}\times\cdots\times G_{p_r}$
	as a subgroup of $G_{p_1+\cdots+p_r}$ 
	for each $p_1,\ldots,p_r\geqslant 0$.
	
	\item \emph{Intersection:}
	We have 
	\[
		(G_{p+q}\times G_r)\cap (G_p\times G_{q+r})
		=
		G_p\times G_q\times G_r,
	\]
	for all $p,q,r\geqslant 0$,
	where $G_{p+q}\times G_r$, $G_p\times G_{q+r}$
	and $G_p\times G_q\times G_r$ are all regarded as
	subgroups of $G_{p+q+r}$.
\end{enumerate}
We denote the neutral element of $G_p$ by $e_p$.
\end{definition}

\begin{remark}
We could delete the intersection 
axiom from Definition~\ref{definition-families},
at the expense of working with the splitting complex
of section~\ref{section-scn} instead of the 
splitting poset.  See Remark~\ref{spnorscn} for further discussion.
\end{remark}

\begin{example}[Symmetric groups]
For $p\geqslant 0$ we let $\Sigma_p$ denote the symmetric
group on $n$ letters.  Then we may form the family of groups
with multiplication $(\Sigma_p)_{p\geqslant 0}$, equipped
with the product maps
\[
	\Sigma_p\times\Sigma_q\to\Sigma_{p+q},
	\qquad
	(f,g)\mapsto f\sqcup g
\]
where $f\sqcup g$ is the automorphism of $\{1,\ldots,p+q\}\cong
\{1,\ldots,p\}\sqcup\{1,\ldots,q\}$
given by $f$ on the first summand and by $g$ on the second.
Then the axioms of a multiplicative family are all immediately verified.
In the case of commutativity, the element $\tau_{pq}$
is the permutation that interchanges the first $p$ and last $q$ letters
while preserving their ordering.
\end{example}

\begin{example}[General linear groups of PIDs]
Let $R$ be a PID.  For $n\geqslant 0$,
let $\GL_n(R)$ denote the general linear group of 
$n\times n$ invertible matrices over $R$.
Then we may form the family of
groups with multiplication $(\GL_p(R))_{p\geqslant 0}$,
equipped with the product maps
\[
	\GL_p(R)\times \GL_q(R)\to\GL_{p+q}(R),
	\qquad
	(A,B)\mapsto \begin{pmatrix} A & 0 \\ 0 & B \end{pmatrix}
\]
given by the block sum of matrices.
The unit, associativity, commutativity, injectivity and intersection axioms 
all hold by inspection.  In the case of commutativity, the
element $\tau_{pq}$ is the permutation matrix
$=\left(\begin{smallmatrix} 0 & I_q \\ I_p & 0 \end{smallmatrix}\right)$.
(It would have been enough to assume that $R$ is a commutative
ring here.  However, as we will see later, we will only be
able to apply our results when $R$ is a PID.)
\end{example}

\begin{example}[Automorphism groups of free groups]
For $p\geqslant 0$ we let $F_p$ denote the free group on $p$ letters,
and we let $\Aut(F_p)$ denote the group of automorphisms of $F_p$.
Then we may form the family of groups with multiplication
$(\Aut(F_p))_{p\geqslant 0}$, equipped with the product maps
\[
	\Aut(F_p)\times \Aut(F_q)
	\to
	\Aut(F_{p+q}),
	\qquad
	(f,g)\mapsto f\ast g.
\]
Here $f\ast g$ is the automorphism of $F_{p+q}\cong F_p\ast F_q$
given by $f$ on the first free factor
and by $g$ on the second.
Then the unit, associativity and connectivity axioms
all hold by inspection.  In the case of commutativity,
the element $\tau_{pq}$ is the automorphism
that interchanges the first $p$ generators with the last $q$ generators.
The injectivity axiom is also clear.  We prove the intersection axiom
as follows.
Suppose that 
$
	f_p\ast f_{q+r} = f_{p+q}\ast f_r
$
where each $f_\alpha$ lies in $\Aut(F_\alpha)$.
We would like to show that $f_{q+r} = f_q\ast f_r$ for some 
$f_q\in\Aut(F_q)$.
Let $x_i$ be one of the middle $q$ generators.
Then $f_{q+r}$ sends $x_i$ to a reduced word in the first $p+q$ generators
and to a reduced word in the last $q+r$ generators.  Since an element
of a free group has a unique reduced expression, it follows that $x_i$ is
sent to a word in the middle $q$ generators.  Thus $f_{q+r} = f_q\ast f_r$
for some $f_q\colon F_q\to F_q$.  By inverting the original equation
we see that in fact $f_q\in\Aut(F_q)$.
\end{example}

\begin{example}[Braid groups]
\label{example-braid}
Given $p\geqslant 0$, let $B_p$ denote the braid group
on $p$ strands. This is defined to be the group of diffeomorphisms
of the disk $D^2$ that preserve the boundary pointwise and
that preserve (not necessarily pointwise) a set 
$X_p\subset D^2$ of $p$ points in the interior
of $D^2$, arranged from left to right,
all taken modulo isotopies relative to $\partial D^2$ and $X_p$.
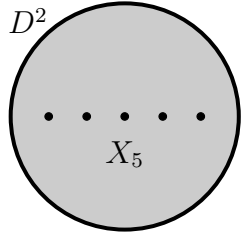
\[\begin{tikzpicture}[scale=0.1]
	\path[draw, line width=1.5, fill=white!80!black] (0,0) circle (15);	
	\path[draw, fill=black] (-10,0) circle (0.5);
	\path[draw, fill=black] (-5,0) circle (0.5);
	\path[draw, fill=black] (-0,0) circle (0.5);
	\path[draw, fill=black] (5,0) circle (0.5);
	\path[draw, fill=black] (10,0) circle (0.5);
	\node at (135:18) {$D^2$};
	\node at (0,-5) {$X_5$};
\end{tikzpicture}\]
The product maps are 
\[
	B_p\times B_q\to B_{p+q},
	\qquad
	(\beta,\gamma)\mapsto \beta\sqcup\gamma
\]
where $\beta\sqcup\gamma$ denotes the braid obtained
by juxtaposing $\beta$ and $\gamma$.  More precisely,
we choose an embedding $D^2\sqcup D^2\hookrightarrow D^2$ 
that embeds two copies of $D^2$ `side by side' in $D^2$,
in such a way that $X_p\sqcup X_q$ is sent into $X_{p+q}$
preserving the left-to-right order.
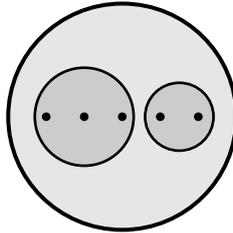
\[\begin{tikzpicture}[scale=0.1]
	\path[draw, line width=1.5, fill=white!90!black] (0,0) circle (15);	
	\path[draw, line width=1, fill=white!80!black] (-5,0) circle (6.5);	
	\path[draw, line width=1, fill=white!80!black] (7.5,0) circle (4.5);	
	\path[draw, fill=black] (-10,0) circle (0.5);
	\path[draw, fill=black] (-5,0) circle (0.5);
	\path[draw, fill=black] (-0,0) circle (0.5);
	\path[draw, fill=black] (5,0) circle (0.5);
	\path[draw, fill=black] (10,0) circle (0.5);
\end{tikzpicture}\]
Then $\beta\sqcup\gamma$ is defined to be the map given by $\beta$
and $\gamma$ on the respective embedded punctured discs, and
by the identity elsewhere.
Then the unit, associativity and injectivity axioms are immediate.
The commutativity axiom holds when we take $\tau_{pq}$ 
to be the class of a diffeomorphism that interchanges the
two embedded discs, passing the left one above the right.  
The intersection axiom follows from the fact
that we may identify the subgroup $B_p\times B_{q+r}\subseteq B_{p+q+r}$
with the set of isotopy classes of diffeomorphisms that fix
an arc that cuts the disc in two,
separating the first $p$ punctures from the last $q+r$ punctures,
and similarly for $B_{p+q}\times B_r$ and $B_p\times B_q\times B_r$.
\end{example}

\section{The splitting poset}
\label{section-spn}

In this section we define the splitting posets associated to
a family of groups with multiplication, and identify
them in the case of symmetric groups, braid groups, general linear groups
of PIDs,
and automorphism groups of free groups.  Conditions on the connectivity
of these posets are the key assumptions in all of our main theorems.

\begin{definition}[The splitting poset]
Let $(G_p)_{p\geqslant 1}$ be a family of groups
with multiplication.
Then for $n\geqslant 2$, 
the $n$-th \emph{splitting poset} $\SP_n$ of $(G_p)_{p\geqslant 1}$
is defined to be the set
\[
	\SP_n=
	\frac{G_n}{G_1\times G_{n-1}}
	\sqcup
	\frac{G_n}{G_2\times G_{n-2}}
	\sqcup
	\cdots
	\sqcup
	\frac{G_n}{G_{n-2}\times G_{2}}
	\sqcup
	\frac{G_n}{G_{n-1}\times G_{1}}
\]
equipped with the partial ordering $\leqslant$
with respect to which
\[
	g(G_{p}\times G_{n-p})
	\leqslant
	h(G_q\times G_{n-q})
\]
if and only if 
$p\leqslant q$ and there is $k\in G_n$ such that 
\[
	g(G_p\times G_{n-p})=k(G_p\times G_{n-p})
	\ \text{ and }\ 
	h(G_q\times G_{n-q})=k(G_q\times G_{n-q}).
\]
Lemma~\ref{chain-lemma} below verifies that the relation
$\leqslant$ is transitive.
\end{definition}

\begin{lemma}\label{chain-lemma}
Given an arbitrary chain
\begin{equation}\label{chain1}
	g_0(G_{p_0}\times G_{n-p_0})
	\leqslant
	g_1(G_{p_1}\times G_{n-p_1})
	\leqslant
	\cdots
	\leqslant
	g_r(G_{p_r}\times G_{n-p_r})
\end{equation}
in $\SP_n$ we may assume, after possibly choosing new
coset representatives, that $g_0=\cdots=g_r$.
It follows that
$g_i(G_{p_i}\times G_{n-p_i})\leqslant g_j(G_{p_j}\times G_{n-p_j})$ 
for any $i\leqslant j$.
\end{lemma}

\begin{proof}
We prove by induction on $s=1,2,\ldots,r$ that given an arbitrary chain
\eqref{chain1} we may assume, after choosing new coset representatives,
that $g_0=\cdots=g_s=g$ for some $g\in G_n$, 
the case $s=r$ being our desired result.

When $s=1$, the claim is immediate from the definition of $\leqslant$.

For the induction step, suppose that the claim holds for $s$.
Take an arbitrary chain~\eqref{chain1} and use the induction hypothesis
to choose new coset representatives so that $g_0=\cdots=g_s=g$.
Since
$g(G_{p_s}\times G_{n-p_s})\leqslant g_{s+1}(G_{p_{s+1}}\times G_{n-p_{s+1}})$
we may assume, after replacing $g_{s+1}$ if necessary, that
$g (G_{p_s}\times G_{n-p_s})=g_{s+1}(G_{p_s}\times G_{n-p_s})$.
Then there are $\gamma\in G_{p_s}$ and $\delta\in G_{n-p_s}$
such that $g^{-1}g_{s+1}=\gamma\diamond\delta$.
Since $e_{p_s}\diamond \delta$ lies in $G_{p_t}\times G_{n-p_t}$
for $t\leqslant s$, we may replace $g$ with $g(e_{p_s}\diamond\delta)$.
And since $\gamma\diamond e_{n-p_s}$ lies in $G_{p_{s+1}}\times G_{n-p_{s+1}}$,
we may replace $g_{s+1}$ with $g_{s+1}(\gamma_{p_s}^{-1}\diamond e_{n-p_s})$.
But then $g_{s+1}=g$. 
So $g_0=\cdots=g_{s+1}$ as required.
\end{proof}

Now we will identify the splitting posets associated to the symmetric groups,
general linear groups of PIDs, 
automorphism groups of free groups, and braid groups.

\begin{proposition}[Splitting posets for symmetric groups]
\label{proposition-spn-sigman}
For the family of groups with multiplication $(\Sigma_p)_{p\geqslant 0}$,
the $n$-th splitting poset $\SP_n$ is isomorphic to the
poset of proper subsets of $\{1,\ldots,n\}$ under inclusion.
\end{proposition}

\begin{proof}
We define a bijection $\phi$ from $\SP_n$ to the poset
of proper subsets of $\{1,\ldots,n\}$ by the rule
\[
	\phi\left(
		g(\Sigma_p\times\Sigma_{n-p})
	\right)
	=
	\{g(1),\ldots,g(p)\}.
\]
This $\phi$ is a well-defined bijection, and we must show that
\[
	g(\Sigma_p\times\Sigma_{n-p})\leqslant h(\Sigma_q\times\Sigma_{n-q})
	\iff
	\{g(1),\ldots,g(p)\}\subseteq\{h(1),\ldots,h(q)\}.
\]
If the first condition holds then $p\leqslant q$ and we may assume that $g=h$, 
so that the second condition follows immediately.
If the second condition holds
then $p\leqslant q$ and, replacing $h$ by $h\circ(k\times \mathrm{Id})$ 
and $g$ by $g\circ (\mathrm{Id}\times l)$
for an appropriate $k\in\Sigma_q$ and $l\in\Sigma_{n-p}$, 
we may assume that $g=h$, so that the first condition holds.
\end{proof}

Let $R$ be a PID.
To identify the splitting posets associated to the family
$(\GL_p(R))_{p\geqslant 0}$, recall that 
Charney in~\cite{Charney} defined $S_R(R^n)$
to be the poset of ordered pairs 
$(P,Q)$ of proper submodules of $R^n$ satisfying $P\oplus Q=R^n$,
equipped with the partial order $\leqslant$ defined by
\[
	(P,Q)\leqslant (P',Q') \iff P\subseteq P'\text{ and }Q\supseteq Q'.
\]
Charney then defined the \emph{split building} of $R^n$,
denoted by $[R^n]$, to be the realisation $|S_R(R^n)|$.
(Note that Charney worked with arbitrary Dedekind domains.)

\begin{proposition}[Splitting posets for general linear groups of PIDs]
\label{proposition-spn-glnr}
Let $R$ be a PID.
For the family of groups with multiplication
$(\GL_n(R))_{n\geqslant 0}$, the splitting poset $\SP_n$ 
is isomorphic to $S_R(R^n)$, so that $|\SP_n|$ is isomorphic
to the split building $[R^n]$.  
\end{proposition}

\begin{proof}
Define $s_1,\ldots,s_{n-1}\in\SP_n$ and $t_1,\ldots,t_{n-1}\in S_R(R^n)$ by
\[s_p=e_n(\GL_p(R)\times\GL_{n-p}(R)),
\qquad
t_p=(\mathrm{span}(x_1,\ldots,x_p),\mathrm{span}(x_{p+1},\ldots,x_n)),\]
where $e_n\in\GL_n(R)$ denotes the identity element and
$x_1,\ldots,x_n$ is the standard basis of $R^n$.
Then the following three properties hold for the elements
$s_i\in\SP_n$, and their analogues hold for the $t_i\in S_R(R^n)$.
\begin{enumerate}
	\item
	\label{ooone}
	$s_1,\ldots,s_{n-1}$ are a complete set of orbit representatives
	for the $\GL_n(R)$ action on $\SP_n$.

	\item
	\label{ootwo}
	The stabiliser of $s_p$ is $\GL_p(R)\times\GL_{n-p}(R)$.
	
	\item
	\label{oothree}
	$x\leqslant y$ if and only if there is $g\in \GL_n(R)$ such that
	$x=g\cdot s_p$ and $y=g\cdot s_q$ where $p\leqslant q$.
\end{enumerate}
It follows immediately that there is a unique isomorphism of posets
$\SP_n\to S_R(R^n)$ satisfying $s_i\mapsto t_i$ for all $i$. 

The three properties hold for $s_i\in\SP_n$ by definition.
We prove them for $t_i\in S_R(R^n)$ as follows.
For \eqref{ooone}, the fact that $R$ is a PID guarantees that
if $(P,Q)\in S_R(R^n)$ then $P$ and $Q$ are free, of ranks $p$ and $q$ say, 
such that $p+q=n$.  If we choose bases of $P$ and $Q$ and concatenate
them to form an element $A\in\GL_n(R)$, then $A\cdot t_p = (P,Q)$
as required.
Property~\eqref{ootwo} is immediate.
For~\eqref{oothree}, suppose that $(P,Q)\leqslant (P',Q')$
and let $p=\mathrm{rank}(P)$ and $p'=\mathrm{rank}(P')$,
so that $p\leqslant p'$.
Then
\[
	R^n = P\oplus (P'\cap Q)\oplus Q',
	\qquad
	P\oplus (P'\cap Q) = P',
	\qquad
	(P'\cap Q)\oplus Q' = Q.
\]
Let $g$ denote the element of $\GL_n(R)$ whose columns are
given by a basis of $P$, followed by a basis of $(P'\cap Q)$, 
followed by a basis of $Q'$.  
Again this is possible since $R$ is a PID.
Then $(P,Q)=g\cdot t_p$ and $(P',Q')=g\cdot t_{p'}$
where $p\leqslant p'$, as required.
\end{proof}

Let us now identify the splitting posets for automorphism groups
of free groups.  The situation is closely analogous to that
for general linear groups.  
Define $S(F_n)$, for each $n\geqslant 2$, 
to be the poset of ordered pairs $(P,Q)$ of 
proper subgroups of $F_n$ satisfying $P\ast Q = F_n$.
It is equipped with the partial order under which
$(P,Q)\leqslant (P',Q')$ if and only if 
$(P,Q)=(J_0,J_1\ast J_2)$ and $(P',Q')=(J_0\ast J_1,J_2)$
for some proper subgroups $J_0,J_1,J_2$ of $F_n$ satisfying
$J_0\ast J_1\ast J_2 = F_n$.
(Note that the condition in the definition of $\leqslant$
is stronger than assuming that $P\subseteq P'$ and $Q'\supseteq Q$).
The proof of the following proposition
is similar to that of Proposition~\ref{proposition-spn-glnr},
and we leave the details to the reader.

\begin{proposition}[Splitting posets for automorphism groups
of free groups]
\label{proposition-spn-autfn}
For the family of groups with multiplication
$(\Aut(F_n))_{n\geqslant 0}$, the splitting poset $\SP_n$ 
is isomorphic to $S(F_n)$.  
\end{proposition}

Let us now identify the splitting posets associated to the
family $(B_p)_{p\geqslant 0}$ of braid groups.  
See Example~\ref{example-braid} for the relevant notation.
Given $n\geqslant 2$, let us
define a poset $\A_n$ as follows.  The elements of $\A_n$
are the arcs embedded in $D^2\setminus X_n$,
starting at the `north pole' of the disc and ending at the `south pole',
such that $X_n$ meets both components of their complement,
all taken modulo isotopies in $D^2\setminus X_n$ that preserve the endpoints.
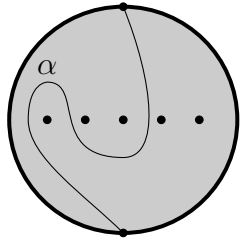
\[\begin{tikzpicture}[scale=0.1]
	\path[draw, line width=1.5, fill=white!80!black] (0,0) circle (15);	
	\path[draw, fill=black] (-10,0) circle (0.5);
	\path[draw, fill=black] (-5,0) circle (0.5);
	\path[draw, fill=black] (-0,0) circle (0.5);
	\path[draw, fill=black] (5,0) circle (0.5);
	\path[draw, fill=black] (10,0) circle (0.5);
	\path[draw, fill=black] (0,15) circle (0.5);
	\path[draw, fill=black] (0,-15) circle (0.5);
	\path[draw] (0,15) .. controls (4,5) and (5,-5) .. (0,-5)
		.. controls (-10,-5) and (-5,5 ) .. (-10,5)
		.. controls (-10,5) and (-12.5,5) .. (-12.5,0)
		.. controls (-12.5,-5) and (-5,-10) .. (0,-15);
	\node at (-10,7) {$\alpha$};
\end{tikzpicture}\]
Given $\alpha,\beta\in\A_n$, we say that $\alpha\leqslant\beta$
if $\alpha$ and $\beta$ have representatives $a$ and $b$
that meet only at their endpoints, and such that
$a$ lies `to the left' of $b$.  (More precisely, $a$ and $b$ must 
meet the north pole in anticlockwise order and the south pole
in clockwise order.)
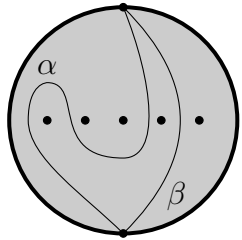
\[\begin{tikzpicture}[scale=0.1]
	\path[draw, line width=1.5, fill=white!80!black] (0,0) circle (15);	
	\path[draw, fill=black] (-10,0) circle (0.5);
	\path[draw, fill=black] (-5,0) circle (0.5);
	\path[draw, fill=black] (-0,0) circle (0.5);
	\path[draw, fill=black] (5,0) circle (0.5);
	\path[draw, fill=black] (10,0) circle (0.5);
	\path[draw, fill=black] (0,15) circle (0.5);
	\path[draw, fill=black] (0,-15) circle (0.5);
	\path[draw] (0,15) .. controls (4,5) and (5,-5) .. (0,-5)
		.. controls (-10,-5) and (-5,5 ) .. (-10,5)
		.. controls (-10,5) and (-12.5,5) .. (-12.5,0)
		.. controls (-12.5,-5) and (-5,-10) .. (0,-15);
	\path[draw] (0,15) .. controls (5,10) and (7.5,5).. (7.5,0)
		.. controls (7.5,-5) and (5,-10) .. (0,-15);
	\node at (-10,7) {$\alpha$};
	\node at (7,-10) {$\beta$};
\end{tikzpicture}\]
Again, the proof of the following is similar to that of
Proposition~\ref{proposition-spn-glnr}, and we leave it
to the reader to provide details if they wish.

\begin{proposition}[Splitting posets for braid groups]
For the family of groups with multiplication $(B_p)_{p\geqslant 0}$,
we have $\SP_n\cong \A_n$.
\end{proposition}

\section{Examples of connectivity of $|\SP_n|$}
\label{section-connectivity}

Our Theorems~\ref{theorem-stability}, \ref{theorem-kernel}
and~\ref{theorem-realisation} apply to a family of groups with
multiplication only when the associated splitting posets
satisfy the connectivity condition that each $|\SP_n|$ is $(n-3)$-connected.
In this section we verify this condition for our main examples:
symmetric groups, where the result is elementary; 
general linear groups of PIDs, where the result
was proved by Charney in~\cite{Charney}; automorphism groups
of free groups, where we make use of Hatcher and Vogtmann's
result on the connectivity of the 
\emph{poset of free factorisations of $F_n$}
in~\cite{HatcherVogtmannCerf}; and for braid groups, where the claim
is a variant of known results on arc complexes.
 
Let us fix our definitions and notation for realisations of posets.
If $P$ is a poset, then its \emph{order complex}
(or \emph{flag complex} or \emph{derived complex}) $\Delta(P)$ 
is the abstract simplicial complex 
whose vertices are the elements of $P$, and in which vertices
$p_0,\ldots,p_r$ span an $r$-simplex if 
they form a chain $p_0<\cdots<p_r$ after possibly reordering.
The \emph{realisation} $|P|$ of $P$ is then defined 
to be the realisation $|\Delta(P)|$ of $\Delta(P)$.
We will usually not distinguish between a simplicial complex
and its realisation.  So if $P$ is a poset, then the simplicial complex
$|P|$ and topological space $|(|P|)|$ will both be denoted
by $|P|$.  
When we discuss topological properties of a poset or
of a simplicial complex, we are referring to the topological properties
of its realisation as a topological space.

\subsection{Symmetric groups}

The result for symmetric groups is elementary.

\begin{proposition}[Connectivity of $|\SP_n|$ for symmetric groups]
For the family of groups with multiplication $(\Sigma_p)_{p\geqslant 0}$
we have $|\SP_n|\cong S^{n-2}$, and in particular
$|\SP_n|$ is $(n-3)$-connected.
\end{proposition}

\begin{proof}
Let $\partial\Delta^{n-1}$ denote the simplicial complex
given by the boundary of the simplex with vertices $1,\ldots,n$.
Then the face poset $\calF(\partial\Delta^{n-1})$ of $\partial\Delta^{n-1}$
is exactly the poset of proper subsets of $\{1,\ldots,n\}$
ordered by inclusion.  But we saw in 
Proposition~\ref{proposition-spn-sigman} 
that the latter is isomorphic to $\SP_n$.
Thus $|\SP_n|\cong|\calF(\partial\Delta^{n-1})|\cong|\partial\Delta^{n-1}|
\cong S^{n-2}$ as required.
\end{proof}

\subsection{General linear groups of PIDs}

Let $R$ be a PID.  In Proposition~\ref{proposition-spn-glnr}
we saw that for the family of groups with multiplication
$(\GL_p(R))_{p\geqslant 0}$ there is an isomorphism
$\SP_n\cong S_R(R^n)$, where $S_R(R^n)$ is the poset whose realisation
is the split building $[R^n]$.
Since $R$ is in particular a Dedekind domain,
Theorem~1.1 of~\cite{Charney} shows that 
$[R^n]$ has the homotopy type of a wedge of $(n-2)$-spheres.
So we immediately obtain the following.

\begin{proposition}[Connectivity of $|\SP_n|$ for general linear groups
of PIDs]
Let $R$ be a PID.
For the family of groups with multiplication
$(\GL_p(R))_{p\geqslant 0}$, and for any $n\geqslant 2$,
$|\SP_n|$ has the homotopy type of a wedge of $(n-2)$-spheres,
and in particular is $(n-3)$-connected.
\end{proposition}

\subsection{Automorphism groups of free groups}
Now we give the proof of the connectivity condition on the splitting
posets for automorphism groups of free groups.  This is the most
involved of our connectivity proofs.

\begin{definition}
Let $F$ be a free group of finite rank.
Define $P(F)$ to be the poset of \emph{ordered}
tuples $H=(H_0,\ldots,H_r)$ of proper subgroups
of $F$ such that $r\geqslant 1$ and $H_0\ast\cdots\ast H_r=F$.
It is equipped with the partial order in which $H\geqslant K$
if $K$ can be obtained by repeatedly amalgamating \emph{adjacent}
entries of $H$.
\end{definition}

\begin{theorem}\label{pf-theorem}
If $F$ has rank $n$,
then $|P(F)|$ has the homotopy type of a wedge of $S^{n-2}$-spheres.
\end{theorem}

\begin{corollary}
[Connectivity of $|\SP_n|$ for automorphism groups of free groups]
\label{spn-autfn-corollary}
For the family of groups with multiplication $(\Aut(F_p))_{p\geqslant 0}$,
the splitting poset $|\SP_n|$ has the homotopy type of a wedge of
$(n-2)$-spheres, and in particular is $(n-3)$-connected.
\end{corollary}

This result has been obtained independently, and with the same proof,
as part of work in progress by Kupers, Galatius and Randal-Williams.  
(See also the remarks after Theorem~\ref{theorem-stability}.)

\begin{proof}[Proof of Corollary~\ref{spn-autfn-corollary}]
If $P$ is a poset then we denote by $P'$ the \emph{derived poset}
of chains $p_0<\cdots<p_r$ in $P$ ordered by inclusion.
Its realisation satisfies $|P'|\cong |P|$.

Recall from Proposition~\ref{proposition-spn-autfn}
that $\SP_n$ is isomorphic to the poset $S(F_n)$
defined there.
So it will suffice to show that $P(F_n)$ is isomorphic to $S(F_n)'$,
for then $|\SP_n|\cong |S(F_n)|\cong |S(F_n)'|\cong |P(F_n)|$
and the result follows from Theorem~\ref{pf-theorem}.
Consider the maps 
\[
	\lambda\colon P(F_n)\to S(F_n)',
	\qquad
	\mu\colon S(F_n)'\to P(F_n)
\]
defined by
\[
	\lambda\Bigl(H_0,\ldots,H_{r+1}\Bigr)=
	\Bigl[ (H_0,H_1\ast \cdots \ast H_{r+1})<\cdots
	< (H_0\ast\cdots\ast H_{r},H_{r+1})\Bigr]
\]
and
\[
	\mu\Bigl[ (A_0,B_0)<\cdots
	< (A_r,B_r)\Bigr]
	=
	\Bigl(A_0,A_1\cap B_0,A_2\cap B_1, \ldots,A_r\cap B_{r-1},B_r\Bigr).
\]
Then one can verify that $\lambda$ and $\mu$ 
are mutually inverse maps of posets.
The verification requires one to use the fact that if
$(X_1,Y_1)<(X_2,Y_2)<(X_3,Y_3)$, then
$X_{1}\ast (X_2\cap Y_{1})=X_2$, $Y_{2}\ast(Y_1\cap X_{2})=Y_1$
and $(X_2\cap Y_{1})\ast(X_{3}\cap Y_2)=X_{3}\cap Y_{1}$,
which follow from the definition of the partial ordering
on $S(F_n)$.
\end{proof}

We now move towards the proof of Theorem~\ref{pf-theorem}.
In order to do so we require another definition.

\begin{definition}
Let $F$ be a free group of finite rank.
Define $Q(F)$ to be the poset of \emph{unordered}
tuples $H=(H_0,\ldots,H_r)$ of proper subgroups
of $F$ such that $r\geqslant 1$ and $H_0\ast\cdots\ast H_r=F$.
Give it the partial order in which $H\geqslant K$
if $K$ can be obtained by repeatedly amalgamating 
entries of $H$, adjacent or otherwise.
Let $f\colon P(F)\to Q(F)$ denote the map that
sends an ordered tuple to the same tuple, now unordered.
\end{definition}

The poset $Q(F_n)$ is exactly the opposite of the
\emph{poset of free factorisations of $F_n$}.
This poset was introduced and studied by Hatcher and Vogtmann
in section~6 of~\cite{HatcherVogtmannCerf}, where it was shown that
its realisation has the homotopy type of a wedge of $(n-2)$-spheres. 
It follows that if $F$ is a free group of rank $m$ then
$|Q(F)|$ has the homotopy type of a wedge of $(m-2)$-spheres.

We will now prove Theorem~\ref{pf-theorem} 
by deducing the connectivity
of $|P(F)|$ from the known connectivity of $|Q(F)|$.
In order to do this we will use a poset fibre theorem
due to Bj\"orner, Wachs and Welker~\cite{BWW}.
Let us recall some necessary notation.
Given a poset $P$ and an element $p\in P$, we define $P_{<p}$
to be the poset $\{q\in P\mid q<p\}$. We define
$P_{\leqslant p}$, $P_{>p}$ and $P_{\geqslant p}$ similarly.
The \emph{length} $\ell(P)$ of a poset $P$ is defined to
be the maximum $\ell$ such that there is a chain
$p_0<p_1<\cdots<p_\ell$ in $P$; the length of the empty poset is
defined to be $-1$.  Theorem~1.1 of~\cite{BWW} states
that if $f\colon P\to Q$ is a map of posets such that
for all $q\in Q$ the fibre $|f^{-1}Q_{\leqslant q}|$ is 
$\ell(f^{-1}Q_{<q})$-connected, then so long as $|Q|$ is connected, we
have
\[
	|P|\simeq
	|Q|\vee\bigvee_{q\in Q} |f^{-1}Q_{\leqslant q}|\ast|Q_{>q}|
\]
where $\ast$ denotes the \emph{join}.
See  the introduction to~\cite{BWW} for further details.

\begin{proof}[Proof of Theorem~\ref{pf-theorem}]
The proof is by induction on the rank of $F$.
When $\rank(F)=2$ we need only observe that
$P(F)$ is an infinite set with trivial partial order,
so that $|P(F)|$ is an infinite discrete set,
and in particular is a wedge of $0$-spheres.
Suppose now that $\rank(F)\geqslant 3$ and that the claim holds for
all free groups of smaller rank than $F$.
Since $\rank(F)\geqslant 3$, $|Q(F)|$ is connected.
Suppose that $H=(H_0,\ldots,H_{r_H})\in Q(F)$.
Then Lemmas~\ref{pf-one}, \ref{pf-two}, \ref{pf-four} and \ref{pf-five}
below tell us the following.
\begin{itemize}
	\item
	$\ell(f^{-1}(Q(F)_{<H}))=r_H-2$
	\item
	$|f^{-1}(Q(F)_{<H})|\cong S^{r_H-1}$
	\item
	$|Q(F)_{>H}|\simeq\bigvee S^{n-r_H-2}$
\end{itemize}
Since $S^{r_H-1}$ is $(r_H-2)$-connected, we may
apply the theorem of Bj\"orner, Wachs and Welker,
which tells us that
\begin{align*}
	|P(F)|
	&\simeq
	|Q(F)|\vee\bigvee_{H\in Q(F)} 
	\bigl(|f^{-1}(Q(F)_{\leqslant H})|\ast|Q(F)_{>H}|\bigr).
	\\
	&\simeq \bigvee S^{n-2}
	\vee\bigvee_{H\in Q(F)}
	\left(\left(\bigvee S^{n-r_H-2}\right)\ast S^{r_H-1}\right)
	\\
	&\simeq \bigvee S^{n-2}
	\vee\bigvee_{H\in Q(F)}\bigvee \left(S^{n-r_H-2}\ast S^{r_H-1}\right)
	\\
	&\simeq \bigvee S^{n-2}\vee\bigvee_{H\in Q(F)}\bigvee S^{n-2}
	\\
	&\simeq\
	\bigvee S^{n-2}
\end{align*}
as required.
\end{proof}

\begin{lemma}
\label{pf-one}
Let $F$ be a free group of finite rank and let $H=(H_0,\ldots,H_r)\in Q(F)$.
Then $\ell(f^{-1}(Q(F)_{<H}))=r-2$.
\end{lemma}

\begin{proof}
After fixing an ordering of the tuple $H$, one can amalgamate
$(r-1)$ adjacent entries before obtaining a $2$-tuple.  This shows that
$\ell(f^{-1}(Q(F)_{\leqslant H}))=r-1$.  Since any maximal chain
must include $H$ itself (with some ordering) it follows that 
$\ell(f^{-1}(Q(F)_{<H}))=r-2$.
\end{proof}

\begin{lemma}
\label{pf-two}
Let $F$ be a free group of finite rank and let $H=(H_0,\ldots,H_r)\in Q(F)$.
Then $|f^{-1}(Q(F)_{\leqslant H})|\cong S^{r-1}$.
\end{lemma}

\begin{proof}
The poset $f^{-1}(Q(F))_{\leqslant H}$ 
is the subposet of $P(F)$ consisting of tuples 
$K=(K_0,\ldots,K_s)$ where each $K_j$ is an amalgamation of some of the $H_i$.  
It is isomorphic to the poset $X_r$ of sequences 
$F=(F_0\subset F_1\subset\cdots\subset F_{s-1})$
of proper subsets of $\{0,\ldots, r\}$, where $F'\leqslant F$
if $F'$ can be obtained from $F$ by forgetting terms of the sequence.
The isomorphism
\[
	X_r\xrightarrow{\ \cong\ }f^{-1}(Q(F))_{\leqslant H}
\]
sends $F=(F_0\subset\cdots\subset F_{s-1})$ to
$K=(K_0,\ldots, K_s)$ where for $i\leqslant s-1$, 
$K_j$ is the subgroup generated
by the $H_i$ for $i\in F_j\setminus F_{j-1}$, and where 
$K_s$ is the subgroup generated by the $H_j$ for $j\not\in F_{s-1}$.
Now $X_n$ is isomorphic to the poset of faces of the barycentric
subdivision of $\partial\Delta^r$, as we see by identifying
$F_0\subset\cdots\subset F_{s-1}$ with the face whose vertices
are the barycentres of the simplices spanned by the $F_i$.
So $|X_n|\cong\partial\Delta^r\cong S^{r-1}$ as claimed.
\end{proof}

Before stating the next lemma we introduce some notation.
Given a poset $P$, let $CP$ denote the poset obtained
by adding a new minimal element $-$.

\begin{lemma}
\label{pf-four}
Let $F$ be a free group of finite rank and let $H=(H_0,\ldots,H_r)\in Q(F)$.
Then
\[
	|Q(F)_{> H}|
	\cong
	|Q(H_0)|\ast \cdots \ast |Q(H_r)|.
\]
\end{lemma}

\begin{proof}
There is an isomorphism
\[
	Q(F)_{\geqslant H}
	\cong
	CQ(H_0)\times \cdots \times CQ(H_r).
\]
It simply takes a tuple $K=(K_0,\ldots,K_s)$
and sends it to the element of $CQ(H_0)\times \cdots \times CQ(H_r)$
whose $CQ(H_i)$-component is the tuple consisting of those $K_j$
which are contained in $H_i$ if there are more than one such,
and which is $-$ otherwise, in which case $H_i$ itself appears
as one of the $K_j$. 
This isomorphism identifies $H$ itself with the tuple $(-,\ldots,-)$,
so that we obtain a restricted isomorphism
\[
	Q(F)_{> H}
	\cong
	CQ(H_0)\times \cdots \times CQ(H_r)\setminus(-,\ldots,-).
\]
Now the realisation of  the right hand side is exactly
$|Q(H_0)|\ast\cdots\ast |Q(H_r)|$, so the result follows.
\end{proof}

\begin{lemma}
\label{pf-five}
Let $F$ be a free group of finite rank and let $H=(H_0,\ldots,H_r)\in Q(F)$.
Then $|Q(H_0)|\ast \cdots \ast |Q(H_r)|$
has the homotopy type of a wedge of $(n-r-2)$-spheres.
\end{lemma}

\begin{proof}
Write $s_i$ for the rank of $H_i$, so that $|Q(H_i)|$
has the homotopy type of a wedge of $(s_i-2)$-spheres.
Since wedge sums commute with joins up to homotopy equivalence,
it follows that $|Q(H_0)|\ast\cdots\ast |Q(H_r)|$
has the homotopy type of a wedge of copies of
$S^{s_0-2}\ast\cdots\ast S^{s_r-2}$.  But then
\[
	S^{s_0-2}\ast\cdots\ast S^{s_r-2}
	\cong
	S^{(s_0-2)+\cdots+(s_r-2)+r}
	=S^{(s_0+\cdots+s_r)-2(r+1) +r}
	=S^{n -r-2}
\]
as required.
\end{proof}

\subsection{Braid groups}
Now we investige the connectivity of the realisations of
the splitting posets for braid groups.  In this case we will
appeal to well-known connectivity results for complexes
of arcs.

\begin{proposition}[Connectivity of $|\SP_n|$ for braid groups]
For the family of groups with multiplication
$(B_p)_{p\geqslant 0}$, and for any $n\geqslant 2$,
$|\SP_n|$ has the homotopy type of a wedge of $(n-2)$-spheres,
and in particular is $(n-3)$-connected.
\end{proposition}

\begin{proof}
Recall from Proposition~\ref{proposition-spn-autfn} 
that we identified $\SP_n$ with the poset of arcs $\A_n$ defined there.
Thus $|\A_n|$ is (the realisation of) the
simplicial complex with vertices the elements of $\A_n$,
in which vertices $\alpha_0,\ldots,\alpha_r$ span a simplex
if and only if, after possibly reordering,
$\alpha_0<\cdots< \alpha_r$.
Now $\alpha_0<\cdots< \alpha_r$ holds if and
only if the $\alpha_i$ have representatives $a_i$ that are
disjoint except at their endpoints, and such that $a_0,\ldots,a_r$
meet the north pole in anticlockwise order.
Thus $|\A_n|$ is the realisation of the
simplicial complex whose vertices are isotopy classes
of nontrivial (they do not separate
a disc from the remainder of the surface) arcs in $D^2\setminus X_n$
from the north pole to the south, where a collection of vertices
form a simplex if they have representing arcs that can be
embedded disjointly except at their endpoints.  
In the notation of section~4 of~\cite{WahlMCG}, this is
exactly the complex $\mathcal{B}(S,\Delta_0,\Delta_1)$
where $S=D^2\setminus X_n$, $\Delta_0\subset\partial D^2$
is the set containing just the north pole,
and $\Delta_1\subset\partial D^2$ is the set containing just the south pole.
Now, replacing $S$ with the complement of $n$ open discs
in $D^2$ does not change the isomorphism type
of the complex.  But in that case, Lemma~4.7 of~\cite{WahlMCG}
applies to show that $|\A_n|$ has connectivity $(n-2)$ greater
than that of $|\A_2|$, which is $(-1)$-connected since it is
a simply a nonempty set.
\end{proof}

\section{Proofs of the applications}
\label{section-applications}
In this section we will assume that
Theorems~\ref{theorem-stability},
\ref{theorem-kernel} and~\ref{theorem-realisation}
hold, and we will prove the remaining theorems stated in the introduction.
We begin with three closely analogous lemmas about the groups
$\GL_n(\Z)$, $\GL_n(\F_2)$ and $\Aut(F_n)$.

\begin{lemma}\label{lemma-low-glnz}
Define elements of $\GL_n(\Z)$, $n=1,2,3$ as follows
\[
	s_1=\begin{pmatrix} -1 \end{pmatrix},
	\qquad
	s_2=\begin{pmatrix} -1 & \\ & 1 \end{pmatrix},
	\qquad
	s_3=\begin{pmatrix} -1 & & \\ & 1 & \\ & & 1\end{pmatrix},
	\qquad
	t=\begin{pmatrix} 1 & 1 \\ 0 & 1\end{pmatrix}.
\]
Use the same symbols to denote the corresponding elements of
$H_1(\GL_n(\Z);\Z) = \GL_n(\Z)_\mathrm{ab}$.
Then the $H_1(\GL_n(\Z);\Z)$ for $n=1,2,3$
are elementary abelian $2$-groups with generators
$s_1\in H_1(\GL_1(\Z);\Z)$, $s_2,t\in H_1(\GL_2(\Z);\Z)$ and $s_3\in H_1(\GL_3(\Z);\Z)$,
and the stabilisation maps have the following effect.
\[\xymatrix@R=5 pt@C=35 pt{
	H_1(\GL_1(\Z);\Z)
	\ar[r]^{s_\ast}	
	&
	H_1(\GL_2(\Z);\Z)	
	\ar[r]^{s_\ast}
	&
	H_1(\GL_3(\Z);\Z)
	\\
	s_1
	\ar@{|->}[r]	
	&
	s_2
	\ar@{|->}[r]	
	&
	s_3
	\\	
	{}
	&
	t\ar@{|->}[r]
	&
	0
}\]
\end{lemma}

\begin{proof}
There are split extensions
\[
	\SL_n(\Z)\longrightarrow \GL_n(\Z)\xrightarrow{\det}\{\pm 1\}
\]
with section determined by $-1\mapsto s_n$, so that we have isomorphisms
\[
	H_1(\GL_n(\Z);\Z)\cong H_1(\SL_n(\Z);\Z)_{\{\pm 1\}}\oplus \Z/2\Z,
\]
where $\Z/2\Z$ is generated by the class of $s_n$.
This isomorphism respects the stabilisation maps.
Now $H_1(\SL_1(\Z);\Z)$ obviously vanishes, and $H_1(\SL_3(\Z);\Z)$ vanishes
since $\SL_n(\Z)$ is perfect for $n\geqslant 3$.  
So it suffices to show that $H_1(\SL_2(\Z);\Z)_{\{\pm 1\}}$
is a group of order $2$ generated by $t$.

Let us write
\[
	u=\begin{pmatrix} 0 & -1 \\ 1 & 1  \end{pmatrix},
	\qquad
	v=\begin{pmatrix} 0 & 1 \\ -1 & 0\end{pmatrix}.
\]
Then $H_1(\SL_2(\Z);\Z)\cong\Z/12\Z$, where $v\leftrightarrow 3$
and $u\leftrightarrow 2$~\cite[p.91]{Knudson}.
One can verify that $s_2vs_2^{-1}=v^{-1}$ and $s_2us_2^{-1}=v^{-1}u^{-1}v$,
so that $\{\pm 1\}$ acts on $H_1(\SL_2(\Z))$ by negation.
Consequently $H_1(\SL_2(\Z);\Z)_{\{\pm 1\}}=(\Z/12\Z)_{\{\pm 1\}}$ 
has order $2$ with generator $t=vu$ as required.
\end{proof}

\begin{lemma}\label{lemma-low-glnftwo}
$H_1(\GL_n(\F_2)) = \GL_n(\F_2)_\mathrm{ab}$ is trivial for $n=1,3$,
and is generated by the element $t$ determined by the matrix
$\left(\begin{smallmatrix} 1 & 1 \\ 0 & 1\end{smallmatrix}\right)$
for $n=2$.
\end{lemma}

\begin{proof}
For $n=1$ this is trivial, and for $n=3$ it follows from the fact
that $\GL_3(\F_2)=\SL_3(\F_2)$ is perfect.  For $n=2$, we simply observe
that $\GL_2(\F_2)$ is a dihedral group of order $6$ generated by
the involutions
\[
\begin{pmatrix} 1 & 1 \\ 0 & 1\end{pmatrix}
\quad\text{and}\quad
\begin{pmatrix} 0 & 1 \\ 1 & 0\end{pmatrix},
\]
so that the abelianization is a group of order $2$ generated
by either of the involutions.
\end{proof}

\begin{lemma}\label{lemma-low-autfn}
Define elements of $\Aut(F_n)$, $n=1,2,3$ as follows.
For $n=1,2,3$ let $s_i$ denote the transformation that inverts
the first letter and fixes the others.
And let $t\in\Aut(F_2)$ denote the transformation
$x_1\mapsto x_1$, $x_2\mapsto x_1x_2$.
Use the same symbols to denote the corresponding elements of
$H_1(\Aut(F_n);\Z) = Aut(F_n)_\mathrm{ab}$.
Then the $H_1(\Aut(F_n);\Z)$ for $n=1,2,3$
are elementary abelian $2$-groups with generators
$s_1\in H_1(\Aut(F_1);\Z)$, $s_2,t\in H_1(\Aut(F_2);\Z)$ 
and $s_3\in H_1(\Aut(F_3);\Z)$,
and the stabilisation maps have the following effect.
\[\xymatrix@R=5 pt@C=35 pt{
	H_1(\Aut(F_1);\Z)
	\ar[r]^{s_\ast}	
	&
	H_1(\Aut(F_2);\Z)	
	\ar[r]^{s_\ast}
	&
	H_1(\Aut(F_3);\Z)
	\\
	s_1
	\ar@{|->}[r]	
	&
	s_2
	\ar@{|->}[r]	
	&
	s_3
	\\	
	{}
	&
	t\ar@{|->}[r]
	&
	0
}\]
\end{lemma}

\begin{proof}
The linearisation map $\Aut(F_n)\to\GL_n(\Z)$ is an isomorphism
on abelianisations for all $n$.  In the case $n=1$ this is
because the map itself is an isomorphism.  In the case $n=2$
this is because the map $\mathrm{Out}(F_2)\to\GL_2(\Z)$ is an
isomorphism, so there is an extension
$F_2\to\Aut(F_2)\to\GL_2(\Z)$ in which the action of 
$\GL_2(\Z)$ on $(F_2)_\mathrm{ab}=\Z^2$ is the tautological one,
so that the coinvariants $((F_2)_\mathrm{ab})_{\GL_2(\Z)}$ vanish,
and the claim follows.  And for $n\geqslant 3$ this is because
$\SL_n(\Z)$ is perfect, as is the subgroup $SA_n$ of $\Aut(F_n)$
consisting of automorphisms with determinant one.
(For the last claim we refer to the presentation of $SA_n$
given in Theorem~2.8 of~\cite{Gersten}.)
The linearisation map sends the generators $s_1,s_2,s_3,t$
listed here to the corresponding generators from Lemma~\ref{lemma-low-glnz},
so the claim follows.
\end{proof}

\begin{proof}[Proof of Theorem~\ref{glnz-outside}]
Let $\F$ be a field of characteristic $2$.
We will use the K\"unneth isomorphism $H_1(-;\F)\cong H_1(-;\Z)\otimes\F$
without further mention.
Theorem~\ref{theorem-kernel} states that the kernel of the map
\begin{equation}
\label{stabilisation-map-one}
	s_\ast\colon H_m(\GL_{2m}(\Z);\F)
	\twoheadrightarrow H_m(\GL_{2m+1}(\Z);\F)
\end{equation}
is the image of the product map
\[
	H_{1}(\GL_{2}(\Z);\F)^{\otimes {m-1}}
	\otimes\ker[H_1(\GL_2(\Z);\F)\xrightarrow{s_\ast} H_1(\GL_3(\Z);\F)]
	\longrightarrow
	H_{m}(\GL_{2m}(\Z);\F).
\]
By Lemma~\ref{lemma-low-glnz},
$H_1(\GL_2(\Z);\F)$ is spanned by the classes $s_2$ and $t$,
and $\ker[H_1(\GL_2(\Z);\F)\xrightarrow{s_\ast} H_1(\GL_3(\Z);\F)]$
is spanned by $t$.
Any product involving both $s_2$ and $t$ vanishes, since
$s_2\cdot t = s_\ast(s_1)\cdot t = s_1\cdot s_\ast(t)=0$.
So it follows that the image of the given product map is precisely
the span of $t^m$, which gives us the claimed description of
of kernel of~\eqref{stabilisation-map-one}.
Next, Theorem~\ref{theorem-realisation} states that the map 
\[
	H_m(\GL_{2m-1}(\Z);\F)\oplus H_1(\GL_2(\Z);\F)^{\otimes m}
	\twoheadrightarrow H_m(\GL_{2m}(\Z);\F)
\]
is surjective.   The second summand of the domain is spanned
by the words in $s_2$ and $t$, 
but the image of any word involving $s_2=s_\ast(s_1)$
lies in the image of $H_m(\GL_{2m-1}(\Z);\F)$.
Thus the image of the given map is in fact spanned by the image
of $H_m(\GL_{2m-1}(\Z);\F)$ and of $t^m$, as required.
\end{proof}

\begin{proof}[Proof of Theorem~\ref{theorem-glnftwo}]
Since $H_m(\GL_{2m+1}(\F_2);\F_2)$ vanishes,
Theorem~\ref{theorem-kernel} shows that $H_m(\GL_{2m}(\F_2);\F_2)$
is spanned by the image of
\[H_1(\GL_2(\F_2);\F_2)^{\otimes(m-1)}\otimes \ker[s_\ast\colon
H_1(\GL_2(\F_2);\F_2)\to H_1(\GL_3(\F_2);\F_2)].\]
But by Lemma~\ref{lemma-low-glnftwo}, this image is precisely the span
of $t^m$.
\end{proof}

\begin{proof}[Proof of Theorem~\ref{theorem-autfn-stability}]
The first claim is immediate from Theorem~\ref{theorem-stability}.
For the second claim, when $\mathrm{char}(\F)\neq 2$ 
we have $H_1(\Aut(F_2);\F)=0$
by Lemma~\ref{lemma-low-autfn}, so that Theorem~\ref{theorem-kernel}
shows that $s_\ast\colon H_\ast(G_{n-1})\to H_\ast(G_n)$ is injective
for $\ast=\frac{n-1}{2}$, and Theorem~\ref{theorem-realisation} shows
that $s_\ast \colon H_\ast(G_{n-1})\to H_\ast(G_n)$ is surjective
for $\ast =\frac{n}{2}$.
\end{proof}

\begin{proof}[Proof of Theorem~\ref{theorem-autfn-outside}]
This is entirely analogous to the proof of Theorem~\ref{glnz-outside},
this time making use of Lemma~\ref{lemma-low-autfn}.
\end{proof}

\section{The splitting complex}
\label{section-scn}

In this section we identify the {realisation} of
the splitting poset $\SP_n$ with the realisation of a semisimplicial
set that we call the `splitting complex'.
It is the splitting complex, rather than the splitting poset,
that will feature in our arguments from this section onwards.
In this section we will make use of semisimplicial sets;
see section~2 of~\cite{RW} for a general discussion of semisimplicial
sets (and spaces) and their realisations.

We have borrowed the name `splitting complex' from work in progress
of Galatius, Kupers and Randal-Williams.  See also the remarks
after Theorem~\ref{theorem-stability}.

\begin{definition}[The splitting complex]\label{sc-definition}
Let $n\geqslant 2$.
The \emph{$n$-th splitting complex} of a family 
of groups with multiplication $(G_p)_{p\geqslant 0}$
is the semisimplicial set $\SC_n$ defined as follows.
Its set of $r$-simplices is
\[
	(\SC_n)_r
	=
	\bigsqcup_{\substack{q_0+\cdots+q_{r+1}=n 
	\\ q_0,\ldots,q_{r+1}\geqslant 1}} 
	\frac{G_n}{G_{q_0}\times\cdots\times G_{q_{r+1}}}
\]
if $r\leqslant n-2$, and is empty otherwise.
And the $i$-th face map
\[
	d_i\colon (\SC_n)_r\longrightarrow(\SC_n)_{r-1},
\]
is defined by
\[
	d_i(g (G_{q_0}\times\cdots\times G_{q_{r+1}}))
	=g (G_{q_0}\times\cdots\times G_{q_i+q_{i+1}}
	\times\cdots\times G_{q_{r+1}})
\]
for $g\in G_n$.
\end{definition}

\begin{example}
Figure~\ref{scfigure} illustrates the splitting complex $\SC_4$.
\begin{figure}\[\xymatrix@C=50pt@R=50pt{
 	\displaystyle\frac{G_4}{G_3\times G_1} 
	&
	\displaystyle\frac{G_4}{G_2\times G_1\times G_1}
	\ar[dl]^(0.3){d_1}\ar[l]_(0.45){d_0}
	\\
	\displaystyle\frac{G_4}{G_2\times G_2}
	& 
	\displaystyle\frac{G_4}{G_1\times G_2\times G_1}
	\ar[ul]|\hole_(0.3){d_0}\ar[dl]|\hole^(0.3){d_1}
	&
	\displaystyle\frac{G_4}{G_1\times G_1\times G_1\times G_1}
	\ar[ul]_(0.3){d_0}\ar[l]_(0.45){d_1}\ar[dl]^(0.3){d_2}
	\\
	\displaystyle\frac{G_4}{G_1\times G_3}
	&
	\displaystyle\frac{G_4}{G_1\times G_1\times G_2}
	\ar[ul]_(0.3){d_0}\ar[l]^(0.45){d_1}
}\]
\caption{The splitting complex $\SC_4$}
\label{scfigure}
\end{figure}
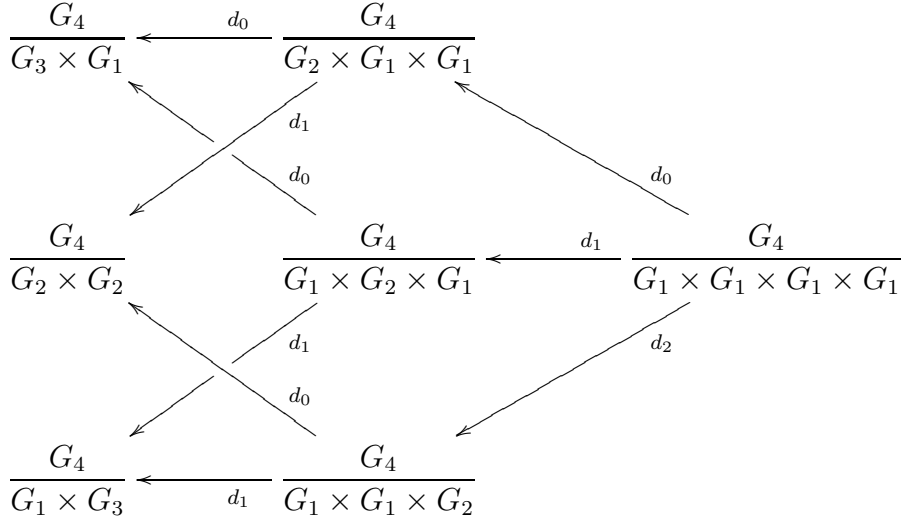
Taking the disjoint union of the terms in each column
gives the $0$-, $1$- and $2$-simplices.
And the arrows leaving each term represent the face maps on that
term, ordered from top to bottom.
\end{example}

\begin{remark}
In the expression $G_{q_0}\times\cdots\times G_{q_{r+1}}$
appearing in Definition~\ref{sc-definition}, we can imagine the
symbols $\times$ as being labelled from $0,\ldots,r$,
so that the $i$-th face map $d_i$ simply
`erases the $i$-th $\times$'.
\end{remark}

Let $P$ be a poset.  The \emph{semisimplicial nerve}
$NP$ of $P$ is defined to be the semisimplicial set whose
$r$-simplices are the chains $p_0<\cdots<p_r$ of length $(r+1)$
in $P$, and whose face maps are defined by
$d_i(p_0<\cdots<p_r) = p_0<\cdots\widehat{p_i}\cdots<p_r$.
The realisation $\|NP\|$ of the semisimplicial nerve is naturally
homeomorphic to the realisation $|P|$ of the poset.

\begin{proposition}
\label{spntoscn}
Let $(G_p)_{p\geqslant 0}$ be a family of groups with multiplication
and let $n\geqslant 2$.
Then $\SC_n\cong N(\SP_n)$.
In particular $|\SP_n|\cong \|\SC_n\|$.
\end{proposition}

\begin{proof}
Let $\phi\colon \SC_n\to N(\SP_n)$ denote the map that sends an $r$-simplex
$g(G_{q_0}\times\cdots\times G_{q_{r+1}})$ of $\SC_n$ to the $r$-simplex
\[	
	g(G_{q_0}\times G_{q_1+\cdots+q_{r+1}})
	<
	g(G_{q_0+q_1}\times G_{q_2+\cdots+q_{r+1}})
	<
	\cdots
	<
	g(G_{q_0+\cdots +q_r}\times G_{q_{r+1}})
\]
of $N(\SP_n)$.  One can verify that $\phi$ is indeed a semi-simplicial map.
Surjectivity follows from Lemma~\ref{chain-lemma}.
Injectivity follows from the fact that
\[
	\bigcap_{i=0}^r G_{q_0+\cdots+q_i}\times G_{q_{i+1}+\cdots+q_{r+1}}
	=
	G_{q_0}\times\cdots\times G_{q_{r+1}},
\]
which follows by induction from the intersection axiom.
\end{proof}

\begin{remark}[Splitting posets or splitting complexes?]
\label{spnorscn}
The results of this section show that if we wish
we could replace $|\SP_n|$ with $\|\SC_n\|$ in the statements of
Theorems~\ref{theorem-stability},
\ref{theorem-kernel} and~\ref{theorem-realisation}.
In doing so, we could jettison the intersection axiom
from Definition~\ref{definition-families}, 
possibly admitting more examples in the process.
However, it is arguably simpler to work with the splitting poset,
and that was certainly the case in sections~\ref{section-spn} 
and~\ref{section-connectivity} where we studied specific examples.
Moreover, the examples of interest to us here all satisfy the intersection
axiom.
We therefore decided to write our paper with 
splitting posets at the forefront.
\end{remark}

\section{A bar construction}
\label{section-bar}

In this section we introduce a variant of the bar construction
which takes as its input an algebra like $\bigoplus_{p\geqslant 0} H_\ast(G_p)$
and produces a graded chain complex (that is, a chain complex of graded
vector spaces) called $\B_n$.
We will see in the next section that $\B_n$ is the $E^1$-term of the
spectral sequence around which all of our proofs revolve.
We fix a field $\F$ throughout.

For the purposes of this section we fix a field $\F$
and a commutative graded $\F$-algebra $A$ equipped with an
additional grading that we call the \emph{charge}.
Thus
\[
	A=\bigoplus_{p\geqslant 0} A_p
\]
where $A_p$ is the part of $A$ with \emph{charge} $p$.
We will call the natural grading of $A$ the \emph{topological}
grading, and we will suppress it from the notation wherever possible.
We require that the multiplication on $A$ respects the charge grading,
and that each charge-graded piece $A_p$ is concentrated in non-negative
degrees. 
We further require that $A_0$ is a copy of $\F$ concentrated in 
topological degree $0$ and (necessarily) generated by the unit element $1$.
In particular, $A$ is augmented.
Finally we assume that $(A_1)_0$, the part of $A$ of charge $1$ and topological
degree $0$, is a copy of $\F$ generated by an element $\sigma$.

\begin{example}
Our only examples of such algebras will be 
\[
	A=\bigoplus_{p\geqslant 0} H_\ast(G_p)
\]
where $(G_p)_{p\geqslant 0}$ is a family of groups with multiplication.
Here the topological grading is the grading of homology,
and the charge grading is obtained from the multiplicative family.
The element $\sigma\in (A_1)_0 = H_0(G_1)$ is defined to be the standard
generator.
\end{example}

\begin{definition}[The chain complex $\B_n$]
Let $A$ be an $\F$-algebra as described at the start of the section.
For $n\geqslant 0$ we define 
$\B_n$ to be the chain complex of {graded}
abelian groups whose $b$-th term is
\[
	(\B_n)_b
	=
	\bigoplus_{\substack{q_0+\cdots+q_b=n\\q_0,\ldots,q_b\geqslant 1}}
	A_{q_0}\otimes\cdots\otimes A_{q_b}
\]
and whose differential is defined by
\[
	d_b(x_0\otimes\cdots\otimes x_b)
	=
	\sum_{i=0}^{b-1}
	(-1)^i x_0\otimes\cdots\otimes x_i\cdot x_{i+1}
	\otimes\cdots\otimes x_b.
\]
For $n=0$ we define $\B_0$ by letting all groups vanish except for $(\B_0)_0$,
which consists of a single copy of $\F$.

Note that $\B_n$ is bigraded.
Its  \emph{homological} grading is the 
grading that is explicit in the definition,
and which is reduced by the differential $d_b$.
Its \emph{topological} grading is the grading obtained
from the topological grading of $A$, and is preserved by the differential $d_b$.
We say that the part of $\B_n$ with homological grading $b$
and topological grading $d$ lies in \emph{bidegree $(b,d)$}.
We reserve the notation $(\B_n)_b$ for the part of $\B_n$
that lies in homological degree $b$.
\end{definition}

\begin{remark}[$\B_n$ and the bar complex]
\label{remark-bar-one}
Regarding $\F$ as a left and right $A$-module \emph{via}
the projection $A\to (A_0)_0=\F$, we may form the two-sided
normalised bar complex $B(\F,\bar A,\F)$
\[
	\F\otimes\F
	\longleftarrow
	\F\otimes\bar A\otimes\F
	\longleftarrow
	\F\otimes\bar A\otimes\bar A\otimes\F
	\longleftarrow
	\F\otimes\bar A\otimes\bar A\otimes\bar A\otimes\F
	\cdots
\]
or, more simply,
\[
	\F
	\longleftarrow
	\bar A
	\longleftarrow
	\bar A\otimes\bar A
	\longleftarrow
	\bar A\otimes\bar A\otimes\bar A
	\longleftarrow
	\cdots
\]
where all tensor products are over $\F$.
This is naturally \emph{trigraded}: there is the homological grading explicit
in the the expressions above, together with charge and topological
gradings inherited from $A$.
Writing $[B(\F,\bar A,\F)]_{\mathrm{charge}=n}$ for the homogeneous piece
with charge grading $n$ inherited from $A$, then we have the following:
\[
	(\B_n)_b = 
	[B(\F,\bar A,\F)_{b+1}]_{\mathrm{charge}=n}.
\]
See Remark~\ref{remark-bar-two} for further discussion.
\end{remark}

\begin{example}
Here is a diagram of $\B_4$.
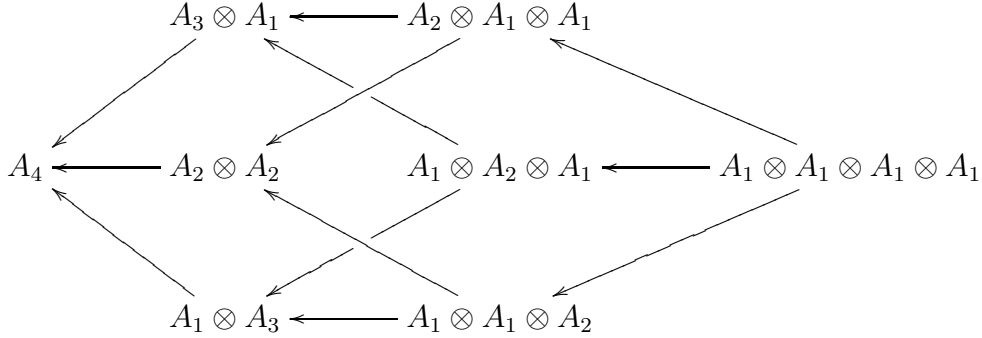
\[\xymatrix@C=40pt@R=40pt{
	&
 	 A_3\otimes A_1\ar[dl]
	&
	 A_2\otimes A_1\otimes A_1\ar[dl]\ar[l]
	\\
	 A_4
	&
	 A_2\otimes A_2\ar[l]
	& 
	 A_1\otimes A_2\otimes A_1
	\ar[ul]|\hole\ar[dl]|\hole
	&
	 A_1\otimes A_1\otimes A_1\otimes A_1
	\ar[ul]\ar[l]\ar[dl]
	\\
	&
	 A_1\otimes A_3\ar[ul]
	&
	 A_1\otimes A_1\otimes A_2\ar[ul]\ar[l]
}\]
The first column of the diagram represents $(\B_4)_0$, the direct sum
of the terms in the next column represent $(\B_4)_1$, and so on.
The effect of the differential $d_b$ on an
element of one of the summands is the alternating sum (taken from
top to bottom) of its images under the arrows exiting that
summand.
The arrows are all constructed using the product of $A$ in the 
evident way. 
\end{example}

\section{The spectral sequence}
\label{section-spectral}

The complex $\B_n$ is our main tool in proving the theorems
stated in the introduction.  The aim of the present section
is to prove the following result, which demonstrates
the connection between $\B_n$ and the splitting poset.
Throughout this section we fix a family of groups with multiplication
$(G_p)_{p\geqslant 0}$ and the algebra $A=\bigoplus H_\ast(G_p)$,
which is of the kind described at the start of section~\ref{section-bar}.
Throughout this section homology is to be taken with coefficients
in an arbitrary field $\F$.

\begin{theorem}
\label{spectral-sequence-theorem}
Let $(G_p)_{p\geqslant 0}$ be a family of groups with multiplication
that satisfies the connectivity axiom, and let
$A=\bigoplus_{p\geqslant 0} H_\ast(G_p)$.
Then there is a first quadrant spectral sequence with $E^1$-term
\[
	(E^1,d^1) = (\B_n,d_b)
\]
for which $E^\infty$ vanishes in bidegrees $(b,d)$ satisfying
$b+d\leqslant (n-2)$.
\end{theorem}

\begin{remark}[The spectral sequence and $\tor$]
\label{remark-bar-two}
In Remark~\ref{remark-bar-one}, we identified $\B_n$ in terms
of a two-sided bar complex.  It follows that we may therefore
identify the $E^2$-term of the above spectral sequence in terms
of a $\tor$ group:
\[
	E^2_{i,j}=
	\tor^A_{i+1}(\F,\F)_
	{\substack{\mathrm{charge}=n\\ \mathrm{topological}=j}}
\]
This observation potentially allows us to use the machinery
of derived functors to understand
the $E^2$-term of our spectral sequence.
We do \emph{not} do this in the present version of this paper.
Instead, our arguments are all done explicitly on the level
of $\B_n$ itself.  
We hope that in a future version of this paper we will 
rephrase our arguments in terms of $\tor$ wherever possible.
\end{remark}

The rest of the section is devoted to the proof of 
Theorem~\ref{spectral-sequence-theorem}.
To begin, we introduce a topological analogue of $\B_n$.
Observe that the multiplication map $G_a\times G_b\to G_{a+b}$
induces a map of classifying spaces $BG_a\times BG_b\to BG_{a+b}$.
We call it the \emph{product map on classifying spaces}
and denote it by $(x,y)\mapsto x\cdot y$.
We will use the product maps on classifying spaces
to create an augmented semisimplicial
space from which we can recover $\B_n$.
See section~2 of~\cite{RW} for conventions about semisimplicial
spaces, augmented semisimplicial spaces, and their realisations.

\begin{definition}[The augmented semisimplicial space $t\B_n$]
Given a family of groups with multiplication $(G_p)_{p\geqslant 0}$,
and given $n\geqslant 2$, we let $t\B_n$ denote the augmented
semisimplicial set whose set of $r$-simplices is given by
\[
	(t\B_n)_r
	=
	\bigsqcup_{\substack{q_0+\cdots+q_{r+1}=n 
	\\ q_0,\ldots,q_{r+1}\geqslant 1}} 
	BG_{q_0}\times\cdots\times BG_{q_{r+1}}
\]
for $r=-1,\ldots,(n-2)$, and which is empty otherwise.
The face map $d_i\colon (t\B_n)_r\to (t\B_n)_{r-1}$ is defined by
\[
	d_i(x_0,\ldots,x_{r+1}) = (x_0,\ldots,x_i\cdot x_{i+1},\ldots,x_{r+1}),
\]
where $\cdot$ denotes the product map on classifying spaces.
\end{definition}

\begin{example}
Here is a diagram of $t\B_4$.
\[\xymatrix@C=40pt@R=40pt{
	&
 	\scriptstyle BG_3\times BG_1\ar[dl]
	&
	\scriptstyle BG_2\times BG_1\times BG_1\ar[dl]\ar[l]
	\\
	\scriptstyle BG_4
	&
	\scriptstyle BG_2\times BG_2\ar[l]
	& 
	\scriptstyle BG_1\times BG_2\times BG_1
	\ar[ul]|\hole\ar[dl]|\hole
	&
	\scriptstyle BG_1\times BG_1\times BG_1\times BG_1
	\ar[ul]\ar[l]\ar[dl]
	\\
	&
	\scriptstyle BG_1\times BG_3\ar[ul]
	&
	\scriptstyle BG_1\times BG_1\times BG_2\ar[ul]\ar[l]
}\]
The four columns correspond to the $r$-simplices of $t\B_4$
for $r=-1,0,1,2$ respectively, the 
disjoint union of the terms in a column being the space of simplices
of the relevant dimension.
\end{example}

The next proposition shows the sense in which $t\B_n$ is a topological
analogue of $\B_n$.

\begin{proposition}[From $t\B_n$ to $\B_n$]
\label{spectral-sequence-proposition}
There is a spectral sequence with $E_1$-term
\[
	(E^1, d^1)
	=
	(\B_n, d_b)
\]
and converging to $H_\ast(\|t\B_n\|)$.
\end{proposition}

\begin{proof}
As in section~2.3 of~\cite{RW}, but with a shift of grading,
the augmented semisimplicial set $t\B_n$
gives rise to a spectral sequence, converging to $H_\ast(\|t\B_n\|)$,
and whose $E^1$-term is given by 
\[
	E^1_{s,t} = H_t((t\B_n)_{s-1}),
\]
with $d^1$ given by the alternating sum of the maps
induced by the face maps of $t\B_n$.
Writing each $(t\B_n)_{s-1}$ as a product of spaces
and applying the K\"unneth isomorphism (which applies
because homology is taken with coefficients in the field $\F$)
we see that this is isomorphic to $\B_n$ equipped with
the differential $d_b$.
\end{proof}

\begin{proposition}
\label{splitting-to-bar}
Suppose that the realisation of the $n$-th splitting
poset $\SP_n$ is $(n-3)$-connected.
Then the realisation $\|t\B_n\|$ is $(n-2)$-connected.
\end{proposition}

\begin{proof}
In order to give this proof, we must be precise about our construction
of classifying spaces.  Given a group $G$, we define $EG$ to be the
realisation of the category obtained from the action of $G$ on itself
by right multiplication.  (So it is $B\overline G$ in the notation
of~\cite{SegalClassifying}.)  Then we define $BG=EG/G$.
The map $EG\to BG$ is a locally trivial principal $G$-fibration,
and $EG$ is itself contractible.
The assignment $G\to EG$ is functorial, and respects products
in the sense that if $G$ and $H$ are groups then the map
$E(G\times H)\to EG\times EH$ obtained from the projections
is an isomorphism.  We can therefore construct a homotopy
equivalence as follows.
\begin{align*}
	BG_{q_0}\times\cdots\times BG_{q_{r+1}}
	&=
	\frac{EG_{q_0}}{G_{q_0}}\times\cdots
	\times\frac{EG_{q_{r+1}}}{G_{q_{r+1}}}
	\\
	&=
	\frac{EG_{q_0}\times\cdots\times EG_{q_{r+1}}}
	{G_{q_0}\times\cdots\times G_{q_{r+1}}}
	\\
	&\xrightarrow{\cong}
	\frac{E(G_{q_0}\times\cdots\times G_{q_{r+1}})}
	{G_{q_0}\times\cdots\times G_{q_{r+1}}}
	\\
	&\xrightarrow{\simeq}
	\frac{EG_n}
	{G_{q_0}\times\cdots\times G_{q_{r+1}}}
\end{align*}
Here the first arrow comes from the compatibility with products.
The second map comes from the iterated product map
$G_{q_0}\times\cdots\times G_{q_{r+1}}\to G_n$,
and it is a homotopy equivalence because it lifts to a map of
principal $(G_{q_0}\times\cdots\times G_{q_{r+1}})$-bundles 
whose total spaces are both contractible.
There is an isomorphism 
\[
	\frac{EG_n}
	{G_{q_0}\times\cdots\times G_{q_{r+1}}}
	\xrightarrow{\ \cong\ }
	EG_n\times_{G_n}\left(
		\frac{G_n}{G_{q_0}\times\cdots\times G_{q_{r+1}}}
	\right)
\]
sending the orbit of an element $x$ to the orbit of 
$(x,e_n(G_{q_0}\times\cdots\times G_{q_{r+1}}))$.
Combining the two maps just constructed gives us a homotopy equivalence:
\begin{equation}\label{equation-borel}
	BG_{q_0}\times\cdots\times BG_{q_{r+1}}
	\xrightarrow{\ \simeq\ }
	EG_n\times_{G_n}\left(
		\frac{G_n}{G_{q_0}\times\cdots\times G_{q_{r+1}}}
	\right)
\end{equation}
Now let $\SC_n^+$ denote the augmented semisimplicial set obtained
from $\SC_n$ by adding a single point as a $-1$-simplex.
The maps~\eqref{equation-borel} then form the components of 
a homotopy equivalence
\[
	(t\B_n)_r\xrightarrow{\ \simeq\ } EG_n\times_{G_n}(\SC_n^+)_r.
\]
These equivalences in turn assemble to a levelwise homotopy equivalence
\[
	t\B_n\xrightarrow{\ \simeq\ }EG_n\times_{G_n}\SC_n^+
\]
and consequently induce a homotopy equivalence
\[
	\|t\B_n\|\xrightarrow{\ \simeq\ }\|EG_n\times_{G_n}\SC_n^+\|.
\]

By assumption, $|\SP_n|$ is $(n-3)$-connected,
so that $\|\SC_n\|$ (to which it is isomorphic by Proposition~\ref{spntoscn})
is also $(n-3)$-connected.  Consequently 
$\|\SC_n^+\|$, which is just the suspension of $\|\SC_n\|$,
is $(n-2)$-connected.
Equivalently, the inclusion of the basepoint
$\ast\hookrightarrow\|\SC_n^+\|$ is an $(n-2)$-equivalence.
It follows that the map
$EG_n\times_{G_n}\ast\to EG_n\times_{G_n}\|\SC_n^+\|$
is also an $(n-2)$-equivalence, so that the quotient
\[
	\frac{EG_n\times_{G_n}\|\SC_n^+\|}
	{EG_n\times_{G_n}\ast}
\]
is $(n-2)$-connected.
But then
\[
	\|t\B_n\|
	\cong
	\|EG_n\times_{G_n}\SC_n^+\|
	\cong
	\frac{EG_n\times_{G_n}\|\SC_n^+\|}
	{EG_n\times_{G_n}\ast}
\]
is $(n-2)$-connected as required.
\end{proof}

\section{Relating $\B_n$ to the stabilisation maps}
\label{section-filtration}

Let $A$ be an $\F$-algebra of the kind described at the start of
section~\ref{section-bar}.  Thus $A$ has a natural topological grading
with respect to which it is commutative, it has an additional
charge grading $A=\bigoplus_{p\geqslant 0}A_p$, $A_0$ consists of
a single copy of $\F$ in topological degree $0$,
$(A_1)_0$ is a copy of $\F$ generated by an element $\sigma$,
and each piece $A_p$ is concentrated in non-negative topological degrees.

\begin{definition}[The stabilisation map.]
The \emph{stabilisation map} $s\colon A_{n-1}\to A_n$ is defined by
$s(a)= \sigma\cdot a$.
\end{definition}

\begin{example}
In the case $A=\bigoplus_{p\geqslant 0} H_\ast(G_p)$ where
$(G_p)_{p\geqslant 0}$ is a family of groups with multiplication,
we take $\sigma$ to be the standard generator of $(A_1)_0=H_0(G_1)$,
and then $s\colon A_{n-1}\to A_n$ is nothing other than the stabilisation
map $s_\ast\colon H_\ast(G_{n-1})\to H_\ast(G_n)$ defined in the
introduction.
\end{example}

The aim of this section is to relate the complex $\B_n$ to the stabilisation
maps.  In order to do so, we introduce complexes $\S_n$ whose homology
quantifies the injectivity and surjectivity of the stabilisation
maps.

\begin{definition}[The complex $\S_n$]
For $n\geqslant 1$, let $\S_n$ denote the graded chain complex
defined as follows.  If $n\geqslant 2$, then $\S_n$ is the 
complex. 
\[\xymatrix@R=5pt{
	(\S_n)_0 && (\S_n)_1\ar[ll]_{d_1}	
	\\
	A_n && A_{n-1}\ar[ll]_{s}
}\]
concentrated in homological degrees $0$ and $1$.
And for $n=1$, $\S_1$ is the complex concentrated in homological
degree $0$, where it is given by the part of $A_1$ lying in positive
degrees, which we denote by $A_{1,>0}$.
\end{definition}

In the case where $A = \bigoplus_{p\geqslant 0}H_\ast(G_p)$
comes from a family of groups with multiplication $(G_n)_{n\geqslant 0}$,
the complex $\S_n$ for $n\geqslant 2$ is simply
\[
	H_\ast(G_n)\xleftarrow{\ \quad s_\ast\quad\ } H_\ast(G_{n-1}),
\]
so that injectivity and surjectivity of the stabilisation map
$s_\ast$ in certain ranges of degrees 
can be expressed as the vanishing of the homology of $\S_n$ in certain
ranges of bidegrees.
All of our results on the stabilisation map are proved from this point
of view.  

Our aim now is to relate the stabilisation maps, \emph{via} the complexes
$\S_n$, to the complex $\B_n$.  We do this using the following filtration.

\begin{definition}\label{filtration-definition}
Given $n\geqslant 2$, define a filtration
\[
	F_0\subseteq F_1\subseteq\cdots\subseteq F_{n-1}=\B_n
\]
of $\B_n$ by defining $F_{n-1}=\B_n$, and by defining
$F_r$ for $r\leqslant (n-2)$ to be the subcomplex
of $\B_n$ spanned by summands of the form
$A_{n-s}\otimes -$ and $A_{1,0}\otimes A_{n-s-1}\otimes-$
for $s\leqslant r$.
As usual $A_{1,0}$ denotes the part of $A$ lying in bidegree $(1,0)$.
Here it is considered as a graded submodule of $A_1$.
\end{definition}

\begin{example}
Let us illustrate the above definition in the case $n=3$,
i.e.~for the filtration $F_0\subseteq F_1\subseteq F_2=\B_3$.
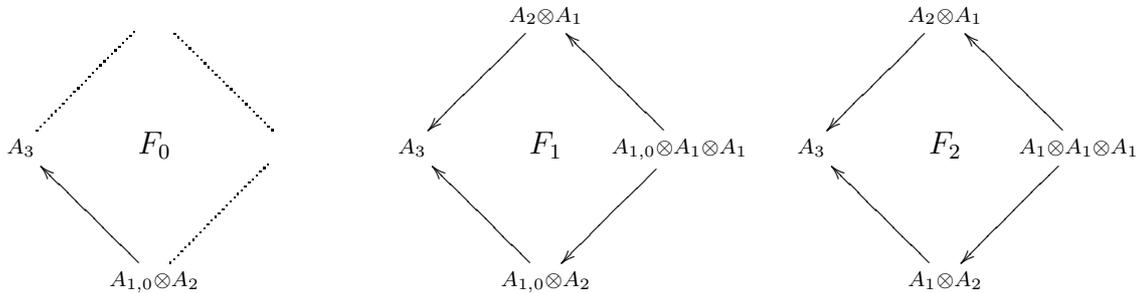
\[\xymatrix@!0@=50pt{
	{}
	&
	\ar@{.}[dl] \phantom{\scriptstyle A_2\otimes A_1}
	&
	{}
	\\
	\scriptstyle A_3
	&
	F_0
	&
	\ar@{.}[ul]\ar@{.}[dl]
	\phantom{\scriptstyle A_1\otimes A_1\otimes A_1}
	\\
	{}
	&
	\scriptstyle A_{1,0}\otimes A_2\ar[ul]
	&
	{}
}
\quad
\xymatrix@!0@=50pt{
	{}
	&
	\scriptstyle A_2\otimes A_1\ar[dl]
	&
	{}
	\\
	\scriptstyle A_3
	&
	F_1
	&
	\scriptstyle A_{1,0}\otimes A_1\otimes A_1
	\ar[ul]\ar[dl]
	\\
	{}
	&
	\scriptstyle A_{1,0}\otimes A_2\ar[ul]
	&
	{}
}
\quad
\xymatrix@!0@=50pt{
	{}
	&
	\scriptstyle A_2\otimes A_1\ar[dl]
	&
	{}
	\\
	\scriptstyle A_3
	&
	F_2
	&
	\scriptstyle A_1\otimes A_1\otimes A_1
	\ar[ul]\ar[dl]
	\\
	{}
	&
	\scriptstyle A_1\otimes A_2\ar[ul]
	&
	{}
}\]
\end{example}

\begin{example}
\label{filtration-example}
In the case $n=4$, we can depict $\B_4$ as follows.
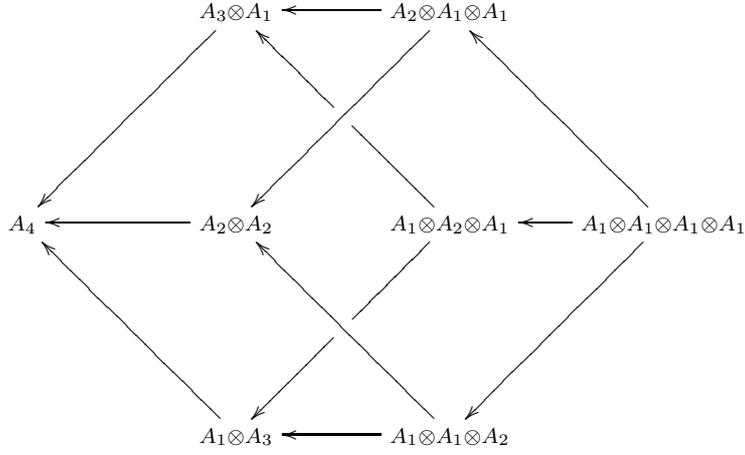
\[\xymatrix@!0@C=80pt@R=80pt{
	&
 	\scriptstyle A_3\otimes A_1\ar[dl]
	&
	\scriptstyle A_2\otimes A_1\otimes A_1
	\ar[dl]\ar[l]
	\\
	\scriptstyle A_4
	&
	\scriptstyle A_2\otimes A_2\ar[l]
	& 
	\scriptstyle A_{1}\otimes A_2\otimes A_1
	\ar[ul]|\hole\ar[dl]|\hole
	&
	\scriptstyle 
	A_{1}\otimes A_1\otimes A_1\otimes A_1
	\ar[ul]\ar[l]\ar[dl]
	\\
	&
	\scriptstyle A_{1}\otimes A_3\ar[ul]
	&
	\scriptstyle A_{1}\otimes A_1\otimes A_2\ar[ul]\ar[l]
}\]
Then we can depict the filtration
\[
	F_0\subseteq F_1\subseteq F_2\subseteq F_3=\B_4
\]
symbolically in the form
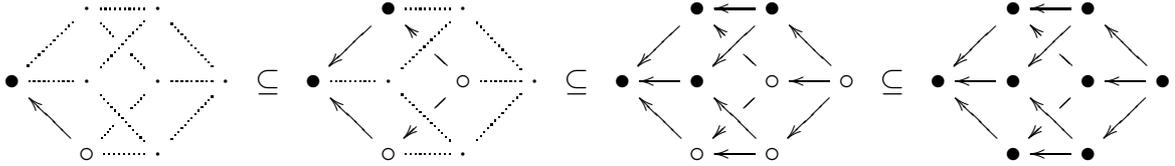
\[
\vcenter{
\xymatrix@=15pt{
	&
 	\cdot\ar@{..}[dl]
	&
	\cdot \ar@{..}[dl]\ar@{..}[l]
	\\
	\bullet
	&
	\cdot\ar@{..}[l]
	& 
	\cdot
	\ar@{..}[ul]|\hole\ar@{..}[dl]|\hole
	&
	\cdot	
	\ar@{..}[ul]\ar@{..}[l]\ar@{..}[dl]
	\\
	&
	\circ\ar[ul]
	&
	\cdot\ar@{..}[ul]\ar@{..}[l]
}}
\ \subseteq\ 
\vcenter{
\xymatrix@=15pt{
	&
 	\bullet\ar[dl]
	&
	\cdot \ar@{..}[dl]\ar@{..}[l]
	\\
	\bullet
	&
	\cdot\ar@{..}[l]
	& 
	\circ
	\ar[ul]|\hole\ar[dl]|\hole
	&
	\cdot	
	\ar@{..}[ul]\ar@{..}[l]\ar@{..}[dl]
	\\
	&
	\circ\ar[ul]
	&
	\cdot\ar@{..}[ul]\ar@{..}[l]
}}
\ \subseteq\  
\vcenter{
\xymatrix@=15pt{
	&
 	\bullet\ar[dl]
	&
	\bullet \ar[dl]\ar[l]
	\\
	\bullet
	&
	\bullet\ar[l]
	& 
	\circ
	\ar[ul]|\hole\ar[dl]|\hole
	&
	\circ	
	\ar[ul]\ar[l]\ar[dl]
	\\
	&
	\circ\ar[ul]
	&
	\circ\ar[ul]\ar[l]
}}
\ \subseteq\ 
\vcenter{
\xymatrix@=15pt{
	&
 	\bullet\ar[dl]
	&
	\bullet \ar[dl]\ar[l]
	\\
	\bullet
	&
	\bullet\ar[l]
	& 
	\bullet
	\ar[ul]|\hole\ar[dl]|\hole
	&
	\bullet	
	\ar[ul]\ar[l]\ar[dl]
	\\
	&
	\bullet\ar[ul]
	&
	\bullet\ar[ul]\ar[l]
}}
\]
where a bullet $\bullet$ indicates that 
the relevant summand of $\B_4$ is included in that term of the filtration, 
a circle $\circ$ indicates a summand $A_1\otimes -$ of $\B_4$
that has been replaced by $A_{1,0}\otimes -$, and 
a dot $\cdot$ indicates an omitted summand.
\end{example}

The next proposition will describe the filtration quotients
of the filtration we have just defined.  In order to state it
we need the following definition.

\begin{definition}
Let $\mathcal{C}$ be a chain complex of graded $\F$-vector spaces
(such as $\B_n$ or $\S_n$).  The \emph{homological suspension} of 
$\mathcal{C}$, denoted $\Sigma_b\mathcal{C}$, is defined
to be the chain complex of graded $\F$-vector spaces
obtained by increasing the homological grading of each term by $1$.
In other words
\[
	(\Sigma_b\mathcal{C})_{b,d} = \mathcal{C}_{b-1,d}
\]
for $b,d\geqslant 0$.
\end{definition}

\begin{proposition}\label{quotient-isomorphism}
For $r\geqslant 1$ there is an isomorphism
\[
	F_r/F_{r-1}
	\cong
	\Sigma_b [\S_{n-r}\otimes \B_r],
\]
while
\[
	F_0\cong \S_n.
\]
\end{proposition}

\begin{example}
Let us illustrate the result of  of Proposition~\ref{quotient-isomorphism}
in the case $n=4$ and $r=2$.  Following on from
Example~\ref{filtration-example}, we see that $F_2/F_1$
can be depicted like this:
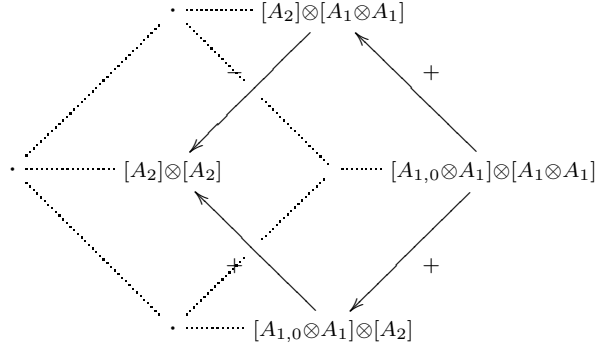
\[\xymatrix@!0@C=60pt@R=60pt{
	&
 	\cdot\ar@{..}[dl]
	&
	\scriptstyle [A_2]\otimes [A_1\otimes A_1]
	\ar[dl]_{-}\ar@{..}[l]
	\\
	\cdot
	&
	\scriptstyle [A_2]\otimes [A_2]\ar@{..}[l]
	& 
	\ar@{..}[ul]|\hole\ar@{..}[dl]|\hole
	&
	\scriptstyle 
	[A_{1,0}\otimes A_1]\otimes [A_1\otimes A_1]
	\ar[ul]_+\ar@{..}[l]\ar[dl]^+
	\\
	&
	\cdot\ar@{..}[ul]
	&
	\scriptstyle [A_{1,0}\otimes A_1]\otimes [A_2]
	\ar[ul]^+\ar@{..}[l]
}\]
The signs on the arrows indicate whether the arrow
is the one obtained from the obvious multiplication map, 
or is the negative of that map.
Observing now that
\[
	\S_2 = (A_2\xleftarrow{\ s\ } A_1)
	\cong (A_2\longleftarrow A_{1,0}\otimes A_1)
\]
and that
\[
	\B_2 = (A_2\longleftarrow A_1\otimes A_1),
\]
where the unmarked arrows are obtained from multiplication maps,
we see that $F_2/F_1$ is isomorphic to the complex depicted as follows.
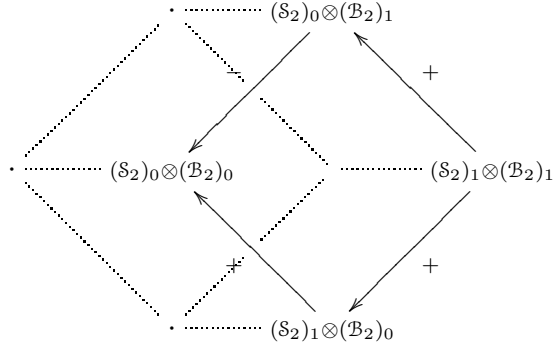
\[\xymatrix@!0@C=60pt@R=60pt{
	&
 	\cdot\ar@{..}[dl]
	&
	\scriptstyle (\S_2)_0\otimes (\B_2)_1
	\ar[dl]_{-}\ar@{..}[l]
	\\
	\cdot
	&
	\scriptstyle (\S_2)_0\otimes (\B_2)_0\ar@{..}[l]
	& 
	\ar@{..}[ul]|\hole\ar@{..}[dl]|\hole
	&
	\scriptstyle 
	(\S_2)_1\otimes (\B_2)_1
	\ar[ul]_+\ar@{..}[l]\ar[dl]^{+}
	\\
	&
	\cdot\ar@{..}[ul]
	&
	\scriptstyle (\S_2)_1\otimes (\B_2)_0
	\ar[ul]^{+}\ar@{..}[l]
}\]
The signs on the arrows now indicate whether the arrow
is equal to the tensor product of a differential from
$\S_2$ or $\B_2$ with an identity map, or to the negative of such.
On the other hand, $\Sigma_b[\S_2\otimes\B_2]$ is exactly the same,
but where now the signs
are governed by the Koszul sign convention.
\[\xymatrix@!0@C=60pt@R=60pt{
	&
 	\cdot\ar@{..}[dl]
	&
	\scriptstyle (\S_2)_0\otimes (\B_2)_1
	\ar[dl]_{+}\ar@{..}[l]
	\\
	\cdot
	&
	\scriptstyle (\S_2)_0\otimes (\B_2)_0\ar@{..}[l]
	& 
	\ar@{..}[ul]|\hole\ar@{..}[dl]|\hole
	&
	\scriptstyle 
	(\S_2)_1\otimes (\B_2)_1
	\ar[ul]_+\ar@{..}[l]\ar[dl]^{-}
	\\
	&
	\cdot\ar@{..}[ul]
	&
	\scriptstyle (\S_2)_1\otimes (\B_2)_0
	\ar[ul]^{+}\ar@{..}[l]
}\]
The last two complexes are isomorphic
\emph{via} the identity map on the summands $(\S_2)_0\otimes (\B_2)_0$
and $(\S_2)_1\otimes(\B_2)_0$, and \emph{via} the negative of the
identity map on the summands $(\S_2)_0\otimes(\B_2)_1$
and $(\S_2)_1\otimes(\B_2)_1$, as claimed in 
Proposition~\ref{quotient-isomorphism}.
\end{example}

\begin{proof}[Proof of Proposition~\ref{quotient-isomorphism}]
For the purposes of the proof, for $m\geqslant 1$ we define
a chain complex of graded $\F$-modules $\bar\S_m$ as follows.
For $m\geqslant 2$, $\bar\S_m$ is
\[\xymatrix@R=5pt{
	\bar\S_0 && \bar\S_1\ar[ll]_{d_1}	
	\\
	A_m && A_{1,0}\otimes A_{m-1}\ar[ll]_{s}
}\]
concentrated in homological degrees $0$ and $1$.
For $m=1$, we define $\bar\S_1$ to be the graded submodule
$A_{1,\geqslant 1}$ of $A_1$ consisting of the terms in positive degree.
Observe that $\bar\S_m$ is isomorphic to $\S_m$ \emph{via}
the identity map $A_m\to A_m$ in homological degree $0$,
and \emph{via} the isomorphism
\[
	A_{1,0}\otimes A_{m-1}\xrightarrow{\cong} A_{m-1},\qquad
	\sigma\otimes x\mapsto x
\]
in homological degree $1$.
We will prove the result with $\bar\S_m$ in place of $\S_m$.

We begin with the case $r\leqslant n-2$.
By definition, $(F_r/F_{r-1})_b$ is the direct sum of the terms
\[
	A_{q_0}\otimes\cdots\otimes A_{q_b}
\]
where $q_0+\cdots+q_b=n$,  $q_1,\ldots,q_b\geqslant 1$,
$q_0=n-r$, together with the terms 
\[
	A_{1,0}\otimes A_{q_1}\otimes \cdots\otimes A_{q_b}
\]
where $1+q_1+\cdots+q_b=n$, $q_1,\ldots,q_b\geqslant 1$, and $q_1=n-r-1$.
In other words, $(F_r/F_{r-1})_b$ is the direct sum of the terms 
\[
	A_{n-r}\otimes[A_{q_0}\otimes\cdots\otimes A_{q_{b-1}}]
\]
where $q_0+\cdots+q_{b-1}=r$,  $q_0,\ldots,q_{b-1}\geqslant 1$,
which is exactly $(\bar\S_{n-r})_0\otimes(\B_r)_{b-1}$,
together with the direct sum of the terms
\[
	[A_{1,0}\otimes A_{n-r-1}] \otimes 
		[A_{q_0}\otimes\cdots\otimes A_{q_{b-2}}]
\]
where $q_0+\cdots+q_{b-2}=r$, $q_0,\ldots,q_{b-2}\geqslant 1$,
which is exactly $(\bar\S_{n-r})_1\otimes(\B_r)_{b-2}$.
But that is exactly 
$(\bar\S_{n-r}\otimes\B_r)_{b-1}=(\Sigma[\bar\S_{n-r}\otimes\B_r])_b$.
Thus we may construct a degree-wise isomorphism
between $F_r/F_{r-1}$ and $\Sigma[\bar\S_{n-r}\otimes\B_r]$
by simply identifying corresponding direct summands.
However, the map constructed this way respects the
differential only up to sign.  To correct this, we map from
$\Sigma_b[\bar\S_{n-r}\otimes\B_r]$ to $F_r/F_{r-1}$ by taking 
$(-1)^{b_2}$ times the identity map on the summands coming from
$(\bar\S_{n-r})_{b_1-1}\otimes (\B_r)_{b_2}$.
One can now check that this gives the required isomorphism
of chain complexes.

The proof in the case $r=n-1$ is similar, and the details are left
to the reader.
\end{proof}

\section{Proof of Theorem~\ref{theorem-stability}}
\label{section-stability}

For the purposes of this section, we let $(G_p)_{p\geqslant 0}$
be a family of groups with multiplication satisfying the
hypotheses of Theorems~\ref{theorem-stability},
\ref{theorem-kernel} and~\ref{theorem-realisation},
and we define $A=\bigoplus_{n\geqslant 0}H_\ast(G_n)$.  
In this section we will prove the following.

\begin{theorem}\label{stability-inductive}
The complexes $\S_n$ for $n\geqslant 1$,
and $\B_n$ for $n\geqslant 2$,
are acyclic in the range $b\leqslant n-2d-1$.
\end{theorem}

Here and in what follows, the phrase ``in the range'' should
be understood to mean ``in the range of bidegrees $(b,d)$ for which''.
So for example,  the theorem states that for $n\geqslant 2$ the complexes
$\S_n$ and $\B_n$ are acyclic in all bidegrees $(b,d)$ for
which $b\leqslant n-2d-1$.

The theorem implies that the homology of $\S_n$ vanishes
in bidegrees $(0,d)$ for $d\leqslant \frac{n-1}{2}$,
and in bidegrees $(1,d)$ for $d\leqslant \frac{n-2}{2}$.
Unwinding the definition of $\S_n$ and $A$, 
we see that this states that $s_\ast\colon H_\ast(G_{n-1})\to H_\ast(G_n)$
is surjective in degrees $\ast\leqslant \frac{n-1}{2}$,
and injective in degrees $\ast\leqslant\frac{n-2}{2}$.
In other words, it exactly recovers the statement of 
Theorem~\ref{theorem-stability}.

Our proof of Theorem~\ref{stability-inductive} will be by
strong induction on $n$.  
The case $n=1$ simply states that the homology of $\S_1$ is 
concentrated in positive degrees, which holds by definition.
The case $n=2$ is immediately verified since it states that the 
maps $s_\ast\colon H_\ast(G_1)\to H_\ast(G_2)$
and $H_\ast(G_1)\otimes H_\ast(G_1)\to H_\ast(G_2)$
are isomorphisms in degree $\ast=0$.
For the rest of the section we will assume that 
Theorem~\ref{stability-inductive} holds for all integers
smaller than $n$, and will will prove that it holds for $n$.

\begin{lemma}
\label{FonetoBn}
Assume that Theorem~\ref{stability-inductive} holds for all
integers smaller than $n$.
Then the composite
\[
	F_1\hookrightarrow F_2\hookrightarrow
	\cdots
	\hookrightarrow F_{n-2}\hookrightarrow F_{n-1}=\B_n
\]
is a surjection on homology in the range $b\leqslant n-2d$
and an isomorphism in the range $b\leqslant n-2d-1$.
\end{lemma}

\begin{proof}
For $r$ in the range $n-1\geqslant r\geqslant 2$, 
the inductive hypothesis tells 
us that $\S_{n-r}$ and $\B_r$ are acyclic 
in the ranges $b\leqslant (n-r)-2d-1$ and $b\leqslant r-2d-1$
respectively. 
Consequently $\S_{n-r}\otimes\B_r$ is acyclic in the range
$b\leqslant n-2d-1$, so that $F_r/F_{r-1}\cong \Sigma_b(\S_{n-r}\otimes\B_r)$
is acyclic in the range $b\leqslant n-2d$.
It follows that $F_{r-1}\to F_r$ is a surjection on homology
in the range $b\leqslant n-2d$ and an isomorphism in the range
$b\leqslant n-2d-1$.

(The estimate for the acyclic range of
$\S_{n-r}\otimes\B_r$ is seen as follows.
The K\"unneth Theorem tells us that
the homology of $\S_{n-r}\otimes\B_r$ is the tensor product
of the homologies of $\S_{n-r}$ and $\B_r$.
Nonzero elements $x$ and $y$ of these respective homologies
must lie in bidegrees $(b_1,d_1)$ and $(b_2,d_2)$
satisfying $b_1\geqslant (n-r)-2d_1$ and $b_2\geqslant r-2d_2$,
so that $x\otimes y$ lies in bidegree $(b_1+b_2,d_1+d_2)$
satisfying $(b_1+b_2)\geqslant n -2(d_1+d_2)$, so that
$\S_{n-r}\otimes\B_r$ is acyclic in the range $b\leqslant n-2d-1$,
as claimed.)
\end{proof}

\begin{lemma}
The inclusion $F_0\hookrightarrow F_1$ is 
an isomorphism in homology
in the range $b\leqslant n-2d-1$.
\end{lemma}

\begin{proof}
Consider the chain complex corresponding to the square
\[\xymatrix@!0@=60pt{
	{}
	&
	A_{n-1}\otimes A_{1,0}\ar[dl]
	&
	{}
	\\
	A_n
	&
	{}
	&
	A_{1,0}\otimes A_{n-2}\otimes A_{1,0}
	\ar[ul]\ar[dl]
	\\
	{}
	&
	A_{1,0}\otimes A_{n-1}\ar[ul]
	&
	{}
}\]
in which the arrows are induced by the multiplication maps
of $A$.
This is a subcomplex $\S q_n$ of $\B_n$, and indeed of $F_1$.
Restricting the filtration $F_0\subset F_1$ of $F_1$ to $\S q_n$
gives a filtration $\bar F_0\subset \bar F_1$ of $\S q_n$ for which
$\bar F_0 = F_0$.  There results a
commutative diagram with short exact rows and
left column an isomorphism.
\[\xymatrix{
	0\ar[r]
	&
	\bar F_0\ar[r]\ar[d]_{\cong}
	&
	\S q_n\ar[r]\ar[d]
	&
	\bar F_1/\bar F_0\ar[r]\ar[d]
	&
	0
	\\
	0\ar[r]
	&
	F_0\ar[r]
	&
	F_1\ar[r]
	&
	F_1/F_0\ar[r]
	&
	0
}\]
The right-hand vertical map is an injection with cokernel
\[
	\Sigma[\S_{n-1}\otimes H_{\ast\geqslant 1}(G_1)].
\]
Since $\S_{n-1}$ is acyclic in the range $b\leqslant (n-1)-2d-1$,
this cokernel is acyclic in the range $b\leqslant [(n-1)-2(d-1)-1]+1=n-2d+1$,
so that the right-hand map in the diagram is a surjection
in homology in the same range.
The connecting homomorphism for the top row is
zero, since $\S q_n$ is isomorphic to the chain complex obtained from the square
\[\xymatrix@!0@=60pt{
	{}
	&
	A_{n-1}\ar[dl]
	&
	{}
	\\
	A_n
	&
	{}
	&
	A_{n-2}
	\ar[ul]\ar[dl]
	\\
	{}
	&
	A_{n-1}\ar[ul]
	&
	{}
}\]
in which each map is multiplication by $\sigma\in A_{1,0}$,
where triviality of the connecting
homomorphism is evident.
The connecting homomorphism for the bottom
sequence is therefore zero in the range (of bidegrees for its domain)
$b\leqslant n-2d+1$.  It follows that in the range $b\leqslant n-2d$
we have short exact sequences 
\[
	0\to H_\ast(F_0)\to H_\ast(F_1)\to H_\ast(F_1/F_0)\to 0.
\]
In the smaller range $b\leqslant n-2d-1$ the third term vanishes,
so that $H_\ast(F_0)\to H_\ast(F_1)$ is an isomorphism as claimed.
\end{proof}

We can now complete the proof of Theorem~\ref{stability-inductive}.
It follows from the last two lemmas
that in the range $b\leqslant n-2d-1$ the 
inclusion $\S_{n}=F_0\hookrightarrow \B_n$ is an isomorphism in homology.
The homology of $\S_n$ is concentrated in the range $b\leqslant 1$,
so that the homology of $\B_n$ vanishes in the range
$2\leqslant b\leqslant n-2d-1$.
It remains to prove that $H_\ast(\S_n)=H_\ast(\B_n)=0$
in the range where $b\leqslant n-2d-1$ and $b\leqslant 1$ both hold.

In order to proceed we use the spectral sequence of
Theorem~\ref{spectral-sequence-theorem},
which has $H_\ast(\B_n)=E^2_{\ast,\ast}$.  
No nonzero differentials $d^r$, $r\geqslant 2$, 
of the spectral sequence affect terms
in the range $b\leqslant n-2d-1$, $b\leqslant 1$.  
This is because any differential with source in this range 
has target outside the first quadrant.  And any differential 
$d^r$ with target in this range has source $E^r_{b+r,d-r+1}$, where
\[
	b+r\leqslant n-2d-1+r \leqslant n-2(d-r+1)-1,
\]
so that $E^r_{b+r,d-r+1}=0$.
Thus $H_\ast(\S_n)=H_\ast(\B_n)=E^\infty_{\ast,\ast}$
in the range $b\leqslant n-2d-1$, $b\leqslant 1$.
Recall that $E^\infty_{\ast,\ast}=0$
in the range $d\leqslant n-2-b$.  
Now 
for $n\geqslant 3$ and $b=0,1$ we have
\[d\leqslant \frac{n-b-1}{2}\implies d\leqslant n-2-b.\]
(The case $d\geqslant 1$ must be treated separately from the case
$d=0$, which is vacuous.)
Thus $H_\ast(\S_n)=H_\ast(\B_n)=E^\infty_{\ast,\ast}=0$
as required.

\section{Proof of Theorem~\ref{theorem-realisation}}
\label{section-realisation}

For the purposes of this section, we let $(G_p)_{p\geqslant 0}$
be a family of groups with multiplication satisfying the
hypotheses of Theorems~\ref{theorem-stability},
\ref{theorem-kernel} and~\ref{theorem-realisation},
and we define $A=\bigoplus_{n\geqslant 0}H_\ast(G_n)$.  
In this section 
we will prove Theorem~\ref{theorem-realisation}, essentially
by extracting a little extra data from
the proof of Theorem~\ref{theorem-stability}, and then exploiting
a cheap trick.
Throughout the section we will write $A_{i,j}$ for the part of $A$
with charge $i$ and topological degree $j$.  In other words,
$A_{i,j}=(A_i)_j$.

\begin{lemma}\label{twomplusone-acyclicity}
For $m\geqslant 1$, the graded chain complex $\B_{2m+1}$
is acyclic in the range $3\leqslant b\leqslant (2m+1)-2d$.
\end{lemma}

\begin{proof}
Lemma~\ref{FonetoBn} shows that the inclusion
$F_1\hookrightarrow \B_n$ is a surjection on homology
in the range $b\leqslant n-2d$.
However, $F_1$ is concentrated in homological degrees
$b=0,1,2$, and so is acyclic in the range $b\geqslant 3$.
Combining the two facts gives the result.
\end{proof}

\begin{lemma}
\label{onemlemma}
In the spectral sequence of 
Theorem~\ref{spectral-sequence-theorem}, for $n=2m+1$,
there are no differentials affecting the term
in bidegree $(1,m)$ from the $E^2$ page onwards.
\end{lemma}

\begin{proof}
Certainly there are no such differentials with
source in this bidegree, since the spectral sequence
is concentrated in the first quadrant.
Since $E^1=\B_{2m+1}$, Lemma~\ref{twomplusone-acyclicity}
shows that $E^2$ vanishes in the range
\[
	3\leqslant b\leqslant (2m+1)-2d.
\]
If $r\geqslant 2$, then
any differential $d^r$ with target in bidegree $(1,m)$
has source in bidegree $(b,d)=(1+r,m-r+1)$,
so that 
\[
	b=(2m+1)-2d - [r-2] \leqslant (2m+1)-2d,
\]
and consequently the source term vanishes.
\end{proof}

\begin{lemma}\label{twomplusoneacyclic-lemma}
Let $m\geqslant 2$.  
Then the complex $\B_{2m+1}$ is acyclic in bidegree $(1,m)$.
\end{lemma}

\begin{proof}
In the spectral sequence of Theorem~\ref{spectral-sequence-theorem}
for $n=2m+1$, we know that $E^2_{1,m}=E^\infty_{1,m}$ 
by Lemma~\ref{onemlemma}, and that $E^\infty_{1,m}=0$
since $m\geqslant 2$ guarantees that $1+m\leqslant (2m+1)-2$.
So $E^2_{1,m}=0$, but this is simply the homology of $\B_{2m+1}$
in bidegree $(1,m)$.
\end{proof}

\begin{lemma}\label{twomacyclic-lemma}
Let $m\geqslant 2$.  Then $\B_{2m}$ is acyclic in bidegree $(0,m)$.
\end{lemma}

\begin{proof}
Consider the following composite.
\[
	\Sigma_b A_{2m}
	\xrightarrow{\ \theta\ }
	\B_{2m+1}
	\xrightarrow{\ \phi\ }
	\Sigma_b \B_{2m}\otimes \B_1
	\xrightarrow{\ \psi\ }
	\Sigma_b\B_{2m}
\]
Here $\theta$ is the map that sends $x\in A_{2m}$
to the element $x\otimes \sigma- \sigma\otimes x\in(\B_{2m+1})_{1}$.
To check that $\theta$ is a chain map, we need only check that the differential
vanishes on its image, which holds because
\[
	d(x\otimes \sigma - \sigma\otimes x)=
	x\cdot \sigma - \sigma\cdot x =0.
\]
Next, 
$\Sigma_b(\B_{2m}\otimes\B_1)$ can be identified with
the submodule of $\B_{2m+1}$ consisting of summands of the form
$-\otimes A_1$, and $\phi$ is the projection onto these summands.
It is a chain map.
Finally, $\psi$ is the map that
projects $\B_1=A_1$ onto its degree $0$ part $A_{1,0}\cong\F$.
In homology in bidegree $(1,m)$ this map is zero since it factors through
the homology of $\B_{2m+1}$, which vanishes in that bidegree.
On the other hand, in this bidegree the composite is simply the suspension
of the map $A_{2m}\to \B_{2m}$, which is a surjection in homological
degree $b=0$.  
It follows that the target of this map, 
which is the homology of $\B_{2m}$ in bidegree
$(0,m)$, is zero.
\end{proof}

\begin{proof}[Proof of Theorem~\ref{theorem-realisation}]
We have seen that $\B_{2m}$ is acyclic in bidegree $(0,m)$.
This means that the map
\[
	\bigoplus_{\substack{p+q=2m \\ p,q\geqslant 1}}
	\bigoplus_{\substack{p'+q'=m \\ p',q'\geqslant 0}}
	A_{p,p'}\otimes A_{q,q'}
	\longrightarrow
	A_{2m,m}
\]
is surjective.  Now, suppose that $p,q,p',q'$ are as in the summation above,
with $p'\leqslant\frac{p-1}{2}$.  Then we have  
the commutative diagram 
\[\xymatrix{
	A_{p,p'}\otimes A_{q,q'}
	\ar[r]
	&
	A_{2m,m}
	\\
	A_{p-1,p'}\otimes A_{q,q'}
	\ar[u]^{s\otimes\mathrm{id}}
	\ar[r]
	&
	A_{2m-1,m}
	\ar[u]_{s}
} \]
in which the left-hand map is surjective by
Theorem~\ref{theorem-stability},
so that the image of
$A_{p,p'}\otimes A_{q,q'}$ is contained in the image of
$s$.  Similarly, if $q'\leqslant\frac{q-1}{2}$, then
the image of $A_{p,p'}\otimes A_{q,q'}$ is contained
in the image of $s$.
The only summands to which these observations do not apply
are those indexed by $p,q,p',q'$ as in the summation,
satisfying also that 
\[
	p'>\frac{p-1}{2},\qquad
	q'>\frac{q-1}{2}.
\]
Adding these inequalities shows that we have
\[
	m=p'+q'>m-1.
\]
Thus the only possibility it that $p'$ is greater than $\frac{p-1}{2}$
by exactly $1/2$, and similarly for $q'$.  In other words,
we must have $p=2p'$ and $q=2q'$.  So we have shown that the map
\[
	A_{2m-1,m}\oplus
	\bigoplus_{\substack{p'+q'=m \\ p',q'\geqslant 1}}
	A_{2p',p'}\otimes A_{2q',q'}
	\longrightarrow
	A_{2m,m}
\]
is surjective.  In the case $m=2$ this proves the claim,
and for $m>2$ the claim now follows by induction.
\end{proof}

\section{Proof of Theorem~\ref{theorem-kernel}}
\label{section-kernel}

For the purposes of this section, we let $(G_p)_{p\geqslant 0}$
be a family of groups with multiplication satisfying the
hypotheses of Theorems~\ref{theorem-stability},
\ref{theorem-kernel} and~\ref{theorem-realisation},
and we define $A=\bigoplus_{n\geqslant 0}H_\ast(G_n)$.  
The aim of this section is to prove
Theorem~\ref{theorem-kernel}, which is an immediate
consequence of Theorem~\ref{theorem-realisation} and the following.

The section will deal with complexes like $\B_n$ which have
a homological and topological grading.  Given such a complex
$\mathcal{C}$, we will write $H_{i,j}(\mathcal{C})$ 
for the part of $H_i(\mathcal{C})$
that lies in topological grading $j$, in other words $H_{i,j}(\mathcal{C})=
H_i(\mathcal{C})_j$.

\begin{theorem}\label{new-kernel-theorem}
Let $m\geqslant 1$.  Then the images of the maps
\begin{gather}
	\ker\Bigl[s_\ast\colon H_{m-1}(G_{2m-2})\to H_{m-1}(G_{2m-1})\Bigr]
	\otimes
	H_1(G_2)
	\longrightarrow
	\ker\Bigl[s_\ast\colon H_{m}(G_{2m})\to H_m(G_{2m+1})\Bigr]
	\label{surjective-map-one}
	\\
	H_{m-1}(G_{2m-2})
	\otimes
	\ker\Bigl[s_\ast\colon H_1(G_2)\to H_1(G_3)\Bigr]
	\longrightarrow
	\ker\Bigl[s_\ast\colon H_{m}(G_{2m})\to H_m(G_{2m+1})\Bigr]
	\label{surjective-map-two}
\end{gather}
together span 
$\ker\left[s_\ast\colon H_{m}(G_{2m})\to H_m(G_{2m+1})\right]$.
\end{theorem}

The main ingredient in the proof of Theorem~\ref{new-kernel-theorem} 
is Lemma~\ref{twomplusoneacyclic-lemma}, 
which states that $H_{1,m}(\B_{2m+1})=0$ for $m\geqslant 1$,
and of which it is an entirely algebraic consequence.
However our argument is significantly more unpleasant than we would like.
Here is the general outline:
Theorem~\ref{new-kernel-theorem} is a statement about $H_{1,m}(\S_{2m+1})$,
which is by definition the kernel 
$\ker\left[s_\ast\colon H_{m}(G_{2m})\to H_m(G_{2m+1})\right]$.
We will use the filtration
\[
	\S_{2m+1}=F_0\subseteq F_1\subseteq\cdots\subseteq F_{2m}=\B_{2m+1}
\]
from Definition~\ref{filtration-definition} to get from
what we know about $H_{1,m}(\B_{2m+1})$ to what we need to know
about $H_{1,m}(\S_{2m+1})$.
We will do this by using the spectral sequence arising from the
filtration in topological degree $m$.
\[
	E^1_{i,j}=H_{i+j,m}(F_i/F_{i-1})
	\implies
	H_{i+j,m}(\B_{2m+1})	
\]
The point is to identify the differentials affecting the term
$E^1_{0,1}=H_{1,m}(\S_{2m+1})$ with the maps~\eqref{surjective-map-one}
and~\eqref{surjective-map-two}.

Let us begin the proof in detail.
We are interested in the values of $H_{r,m}(F_i/F_{i-1})$
in the cases $r=0,1,2$.  
Recall from Proposition~\ref{quotient-isomorphism}
that for $i\geqslant 1$ we have
\[
	F_i/F_{i-1}
	\cong
	\Sigma_b [\S_{2m+1-i}\otimes\B_i]
\]
so that 
\[
	H_{r,m}(F_i/F_{i-1})
	\cong
	H_{r-1,m}[\S_{2m+1-i}\otimes\B_i]
	\cong
	\bigoplus_{\substack{r_1+r_2=r-1\\ m_1+m_2=m}}
	H_{r_1,m_1}(\S_{2m+1-i})\otimes H_{r_2,m_2}(\B_i).
\]
We have the following.

\begin{lemma}
For $r=0,1,2$ and $i=0,\ldots,2m$,
the only nonzero groups $H_{r,m}(F_i/F_{i-1})$ are as follows.
\begin{align*}
	H_{1,m}(F_0) & \cong H_{1,m}(\S_{2m+1}) \\
	H_{2,m}(F_0) & \cong H_{2,m}(\S_{2m+1}) \\
	H_{1,m}(F_1/F_0) & \cong H_{0,m}(\S_{2m})\otimes H_{0,0}(\B_1)\\
	H_{2,m}(F_1/F_0) & \cong H_{1,m}(\S_{2m})\otimes H_{0,0}(\B_1)\\
	H_{2,m}(F_2/F_1) & \cong H_{1,m-1}(\S_{2m-1})\otimes H_{0,1}(\B_2)\\
	H_{2,m}(F_3/F_2) & \cong H_{0,m-1}(\S_{2m-2})\otimes H_{1,1}(\B_3)
\end{align*}
\end{lemma}

\begin{proof}
{\bf Case $i=0$.}
In this case we have $H_{r,m}(F_0)=H_{r,m}(\S_{2m+1})$,
and by Theorem~\ref{stability-inductive} this
is nonzero only for $r\geqslant 1$.

{\bf Case $i=1$.}
In this case we have
\[
	H_{r,m}(F_1/F_0)
	\cong H_{r,m}(\Sigma_b[\S_{2m}\otimes\B_1])
	\cong H_{r-1,m}(\S_{2m}\otimes\B_1)
	\cong \bigoplus_{m_1+m_2=m} H_{r-1,m_1}(\S_{2m})\otimes H_{0,m_2}(\B_1)
\]
since $\B_1$ is concentrated in homological degree $b=0$.
Now by Theorem~\ref{stability-inductive} the term 
$H_{r-1,m_1}(\S_{2m})$ vanishes for
$m_1\leqslant m-r/2$.  So for $r=0$ we require $m_1>m$, 
which is impossible, and for $r=1,2$ the only possibility is $m_1=m$, $m_2=0$.
So  the possible terms are
\[
H_{1,m}(F_1/F_0)\cong H_{0,m}(\S_{2m})\otimes H_{0,0}(\B_1)
\]
and
\[
H_{2,m}(F_1/F_0)\cong H_{1,m}(\S_{2m})\otimes H_{0,0}(\B_1).
\]

{\bf Case $2\leqslant i\leqslant 2m$.}
In this case we have
\begin{align*}
	H_{r,m}(F_i/F_{i-1})
	&\cong H_{r,m}(\Sigma_b[\S_{2m+1-i}\otimes\B_i])
	\\
	&\cong H_{r-1,m}(\S_{2m+1-i}\otimes\B_i)
	\\	
	&\cong \bigoplus_{\substack{r_1+r_2=r-1\\m_1+m_2=m}}
	H_{r_1,m_1}(\S_{2m+1-i})\otimes H_{r_2,m_2}(\B_i).
\end{align*}
Now from Theorem~\ref{stability-inductive} we know that
$H_{r_2,m_2}(\B_i)=0$ for $r_2\leqslant 2i-2m_2-1$ 
while $H_{r_1,m_1}(\S_{2m+1-i})=0$ for $r_1\leqslant 2m+1-i-2m_1-1$
Thus a nonzero group appearing in the direct sum above must have 
\[
	r_1=2m-i-2m_1+\delta
	\text{ and }
	r_2=i-2m_2-1+\epsilon
\]
for $\delta,\epsilon>0$. 
Then the constraints $r_1+r_2=r-1$ and $m_1+m_2=m$ give us
$r=\delta+\epsilon$.  Thus, to find a nonzero group when
$i\geqslant 2$ and $r=0,1,2$, the only possibility is that
$r=2$ and $\delta=\epsilon=1$.  But then $(r_1,r_2)=(1,0)$ or
$(r_1,r_2)=(0,1)$, in which case we have two possible summands, 
only one of which is possible at a given time, namely
\[
	H_{2,m}(F_i/F_{i-1})
	=
	\left\{\begin{array}{ll}
	H_{0,m-(i-1)/2}(\S_{2m+1-i})\otimes H_{1,(i-1)/2}(\B_i)
	&
	\text{for }i\text{ odd,}
	\\
	H_{1,m-i/2}(\S_{2m+1-i})\otimes H_{0,i/2}(\B_i)
	&
	\text{for }i\text{ even.}
	\end{array}\right.
\]
However, Lemmas~\ref{twomplusoneacyclic-lemma} and~\ref{twomacyclic-lemma}
guarantee that the second factors vanish for $i\geqslant 4$.
Thus the only contributing factors are
\[
	H_{2,m}(F_3/F_2)
	=
	H_{0,m-1}(\S_{2m-2})\otimes H_{1,1}(\B_3)
\]
and
\[
	H_{2,m}(F_2/F_1)=H_{1,m-1}(\S_{2m})\otimes H_{0,1}(\B_2).
\]
This completes the proof.
\end{proof}

Thus the spectral sequence is as follows.
\[
\begin{tikzpicture}
  \matrix (m) [matrix of math nodes,
    nodes in empty cells,nodes={minimum width=5ex,
    minimum height=5ex,outer sep=-5pt},
    column sep=1ex,row sep=1ex]{
\scriptstyle H_{2,m}(\S_{2m+1}) &  \bullet   &   \bullet  & \bullet & \bullet \\
\scriptstyle H_{1,m}(\S_{2m+1}) & \scriptstyle H_{1,m}(\S_{2m})\otimes H_{0,0}(\B_1)   &  \bullet & \bullet & \\
0  & \scriptstyle H_{0,m}(\S_{2m})\otimes H_{0,0}(\B_1)& \scriptstyle  H_{1,m-1}(\S_{2m-1})\otimes H_{0,1}(\B_2)  &  \bullet& \bullet \\
0 & 0 & 0 & \scriptstyle H_{0,m-1}(\S_{2m-2})\otimes H_{1,1}(\B_3)& \bullet \\
0 & 0 & 0 & 0 & 0 \\
	};
  \draw[dotted] (m-1-1) -- (m-2-1);
  \draw[dotted] (m-2-1) -- (m-3-1);
  \draw[dotted] (m-3-1) -- (m-4-1);
  \draw[dotted] (m-4-1) -- (m-5-1);
  \draw[dotted] (m-3-1) -- (m-3-2);
  \draw[dotted] (m-3-2) -- (m-3-3);
  \draw[dotted] (m-3-3) -- (m-3-4);
  \draw[dotted] (m-3-4) -- (m-3-5);
\end{tikzpicture}
\]
We will now investigate the differentials
affecting $H_{1,m}(\S_{2m+1})$.
We will need the following preliminary result.

\begin{lemma}
\label{Bthreelemma}
An arbitrary element of $H_{1,1}(\B_3)$
has a representative of the form
\[
	(x\otimes \sigma - \sigma\otimes x) + q\otimes\sigma
\]
where $x\in A_{2,1}$ and $q\in\ker(s\colon A_{2,1}\to A_{3,1})$,
and $\sigma\in A_{1,0}$ is the stabilising element.
\end{lemma}

\begin{proof}
$A_1$ and $A_2$ are concentrated in non-negative degrees,
and in degree $0$ they are spanned by $\sigma$ and $\sigma^2$
respectively.  Thus an arbitrary cycle of $\B_3$ in bidegree
$(1,1)$ has form $j\otimes\sigma + k\otimes\sigma^2 + \sigma\otimes l
+\sigma^2\otimes m$ for $j,l\in A_{2,1}$ and $k,m\in A_{1,1}$.
By adding $d(k\otimes\sigma\otimes\sigma -\sigma\otimes\sigma\otimes m)$,
we may assume that $k=m=0$, so that our cycle has the form
$j\otimes\sigma+\sigma\otimes l$.
This can be rewritten in the required form with $x=-l$
and $q=j+l$.
\end{proof}

\begin{lemma}
The span of the images of the differentials 
with target $H_{1,m}(\S_{2m+1})\cong
\ker(s\colon H_m(G_{2m})\to H_m(G_{2m+1})$ is precisely
the span of the maps \eqref{surjective-map-one}
and \eqref{surjective-map-two}.
\end{lemma}

\begin{proof}
There are just three positions in the spectral sequence
supporting differentials with the given target.
We will compute the differentials case by case.

{\bf Case 1: 
Differentials with domain
$H_{1,m}(\S_{2m})\otimes H_{0,0}(\B_1)$.}
An element $l$ of the domain 
can be represented by a cycle  $l_1=x\otimes\sigma$ in $\S_{2m}\otimes\B_1$,
where $x\in\ker(s\colon A_{2m,m}\to A_{2m+1,m})$ and $\sigma\in A_{1,0}$
is the stabilising element.  
Then under the isomorphism of Proposition~\ref{quotient-isomorphism},
$l_1$ corresponds to the element $l_2=\sigma\otimes x\otimes \sigma$
of $F_1/F_0$.  We lift this to the element
$l_3=\sigma\otimes x \otimes \sigma$ of $F_1$.
Then $d(l_3) = \sigma\cdot x \otimes \sigma - \sigma\otimes x\cdot \sigma=0$.
Thus all differentials $d^r$ vanish on $l$.  (In fact there is only
one possibility, $d^1$.)

{\bf Case 2: Differentials with domain
$H_{1,m-1}(\S_{2m-1})\otimes H_{0,1}(\B_2)$}.
An element $l$ of the domain can be represented by a linear combination
of cycles of the form $x\otimes y$ in $\S_{2m-1}\otimes\B_2$,
where $x\in\ker(s\colon A_{2m-2,m-1}\to A_{2m-1,m-1})$
and $y\in A_{2,1}$.  Let us assume without loss that $l$
is in fact represented by $l_1=x\otimes y$.
Then under the isomorphism of Proposition~\ref{quotient-isomorphism},
$l_1$ corresponds to the element $l_2=\sigma\otimes x\otimes y$
of $F_2/F_1$, which we lift to the element $l_3=\sigma\otimes x\otimes y$
of $F_2$.  Now $d(l_3) = \sigma\cdot x\otimes y - \sigma \otimes x\cdot y
=-\sigma\otimes x\cdot y$, which lies in $F_0$.  
Thus $d^1(l)=0$, while $d^2(l)$ is the class represented by 
$-\sigma\otimes x\cdot y$, which under the isomorphism
of Proposition~\ref{quotient-isomorphism} corresponds to the element
$-x\cdot y$ of $A_{2m,m} = H_m(G_{2m})$.  This is precisely the image
of $-x\cdot y$ under the map~\eqref{surjective-map-one} above.
Thus the image of $d^2$ is precisely the image of~\eqref{surjective-map-one}.

{\bf Case 3:
Differentials with domain $H_{0,m-1}(\S_{2m-2})\otimes H_{1,1}(\B_3)$.}
By Lemma~\ref{Bthreelemma}, an element $l$ of the domain has a representative of the form
\[
	l_1=\sum_\alpha x_\alpha
	\otimes(y_\alpha\otimes\sigma-\sigma\otimes y_\alpha)
	+\sum_\beta p_\beta\otimes(q_\beta\otimes\sigma)
\]
where $x_\alpha,p_\beta\in A_{2m-2,m-1}$, $y_\alpha\in A_{2,1}$
and $q_\beta\in \ker(s\colon A_{2,1}\to A_{3,1})$.
Under the isomorphism of Proposition~\ref{quotient-isomorphism},
$l_1$ corresponds to the element
\[
	l_2=\sum_\alpha ( x_\alpha
	\otimes y_\alpha\otimes\sigma- x_\alpha\otimes \sigma\otimes y_\alpha)
	+\sum_\beta p_\beta\otimes q_\beta\otimes\sigma
\]
of $F_3/F_2$.  We lift this to the element
\[
	l_3=\sum_\alpha ( x_\alpha
	\otimes y_\alpha\otimes\sigma- x_\alpha\otimes \sigma\otimes y_\alpha
	+\sigma\otimes x_\alpha\otimes y_\alpha)
	+\sum_\beta p_\beta\otimes q_\beta\otimes\sigma
\]
of $F_3$. (The apparently new terms lie in $F_2$.)
Then
\[
	d(l_3)=\sum_\alpha ( x_\alpha\cdot y_\alpha\otimes \sigma
	- \sigma\otimes x_\alpha\cdot y_\alpha)
	+\sum_\beta p_\beta\cdot q_\beta\otimes\sigma.
\]
This lies in $F_1$, so that $d^1(l)=0$, and its image in $F_1/F_0$ is
\[
	\sum_\alpha  x_\alpha\cdot y_\alpha\otimes \sigma
	+\sum_\beta p_\beta\cdot q_\beta\otimes\sigma
\]
so that applying the isomorphism of Proposition~\ref{quotient-isomorphism}
shows that
\[
	d^2(l) = \left[\sum_\alpha x_\alpha\cdot y_\alpha
	+ \sum_\beta p_\beta\cdot q_\beta\right] \otimes [\sigma]
	\in H_{0,m}(\S_{2m})\otimes H_{0,0}(\B_1).
\]
Thus $l$ lies in the kernel of $d^2$ if and only if
\[
	\left[\sum_\alpha x_\alpha\cdot y_\alpha
	+ \sum_\beta p_\beta\cdot q_\beta\right]
\]
is zero in $H_{0,m}(\S_{2m})$, or in other words if and only if there
is $w\in A_{2m-1,m}$ such that 
$\sum_\alpha x_\alpha\cdot y_\alpha
	+ \sum_\beta p_\beta\cdot q_\beta= \sigma\cdot w$.
In this case, we may again represent $l$ by $l_1$,
which again corresponds to the element $l_2$ of $F_3/F_2$,
but which we now lift to the element $l_3-\sigma\otimes w\otimes\sigma$
of $F_3$.  (The additional term lies in $F_1$.)
But then $d(l_3-\sigma\otimes w\otimes \sigma)$
is precisely the element
\[
	\sum_\beta \sigma\otimes p_\beta\cdot q_\beta
\]
of $F_0$.  Applying the isomorphism of Proposition~\ref{quotient-isomorphism},
we find that
\[
	d^3(l) = \left[\sum_\beta p_\beta\cdot q_\beta\right]
	\in H_{1,m}(\S_{2m+1}).
\]
But then it follows that the image of $d^3$ is precisely the span
of the map \eqref{surjective-map-two}.
\end{proof}

We may now complete the proof.  Since $H_{1,m}(\B_{2m+1})=0$,
it follows that the infinity-page of the spectral sequence
must vanish in total degree $1$.  So then in particular we must have
$E^\infty_{1,0}=0$, or in other words, the differentials with
target $H_{1,m}(\S_{2m+1})$ must span.  But we have identified the
(nonzero) differentials with the maps \eqref{surjective-map-one}
and \eqref{surjective-map-two}.  Thus it follows that together,
the images of these two maps must span.
This completes the proof of Theorem~\ref{new-kernel-theorem}.

\bibliographystyle{plain}
\bibliography{multiplicative}
\end{document}